\documentclass[preprint,3p,11pt,times]{elsarticle}

\usepackage{hyperref,color}
\usepackage{amssymb}
\usepackage{amsthm,amsmath}
\usepackage{enumerate}
\biboptions{sort&compress,square,numbers}

\numberwithin{equation}{section}
\newtheorem{theorem}{Theorem}[section]

\newtheorem{corollary}[theorem]{Corollary}
\newtheorem{lemma}[theorem]{Lemma}
\theoremstyle{remark}

\newtheorem{remark}[theorem]{Remark}
\newtheorem{assumption}{Assumption}[section]

\newcommand{\jump}{\mathcal{P}}
\newcommand{\life}{t_+}
\newcommand{\nlife}{t_{-}}
\newcommand{\lifepm}{t_{\pm}}
\newcommand{\lifemp}{t_{\mp}}
\newcommand{\exstate}{E_\Delta}
\newcommand{\gstate}{\widetilde{E}}
\newcommand{\Leb}{\mathrm{Leb}}
\newcommand{\Ker}{\mathrm{Ker}}

\newcommand{\sQ}{\mathfrak{Q}}
\newcommand{\dive}{\operatorname{div}}
\newcommand{\Tm}{\mathrm{T}_{\max}}
\newcommand{\F}{B}
\newcommand{\Lip}{\operatorname{Lip}}
\newcommand{\bound}{\Gamma^{\pm}}
\newcommand{\boundn}{\Gamma^{\mp}}
\newcommand{\pbound}{\Gamma^{+}}
\newcommand{\nbound}{\Gamma^{-}}

\renewcommand{\vartheta}{q}
\renewcommand{\gamma}{\mathrm{Tr}}
\renewcommand{\Phi}{F}
\renewcommand{\beta}{\eta}

\begin{document}

\begin{frontmatter}

\title{Densities for piecewise deterministic Markov processes with boundary\tnoteref{label1}}
\author[a]{Piotr Gwi\.zd\.z}
\ead{piotr.gwizdz@gmail.com}

\author[a,b]{Marta Tyran-Kami\'nska\corref{cor1}}
\cortext[cor1]{Corresponding author}
\ead{mtyran@us.edu.pl}
\address[a]{Institute of Mathematics, University of Silesia, Bankowa 14, 40-007
Katowice, Poland}
\address[b]{Institute of Mathematics, Polish Academy of Sciences, Bankowa 14, 40-007
Katowice, Poland}

\date{\today}
\tnotetext[label1]{This work was partially supported  by the Polish NCN grant  2017/27/B/ST1/00100 and by the grant 346300 for IMPAN from the Simons Foundation and the matching 2015-2019 Polish MNiSW fund.}

\begin{keyword}
substochastic semigroup\sep invariant density\sep perturbation of boundary conditions\sep initial-boundary value problem\sep transport equation\sep  cell cycle model
\MSC[2010] 47D06\sep 60J25\sep 60J35
\end{keyword}

\begin{abstract}
We study the  existence of densities for distributions of piecewise deterministic Markov processes. We also obtain relationships between invariant densities of the continuous time process and that of the process observed at jump times. In our approach we use functional-analytic methods and the theory of linear operator semigroups.
By  imposing general conditions on the characteristics of a given Markov process, we show the existence of a substochastic semigroup describing the evolution of densities for the process and we identify its generator. Our main tool is a new perturbation theorem for substochastic semigroups, where we perturb  both the action of the generator and of its domain, allowing
to treat general transport-type equations with non-local boundary conditions.
A couple of particular examples illustrate our general results.
\end{abstract}

\end{frontmatter}
\section{Introduction}\label{s:intro}

Piecewise deterministic Markov processes (PDMPs) were introduced by Davis \cite{davis84} as stochastic models involving deterministic motions and random jumps.  The sample paths of a PDMP $\{X(t)\}_{t\ge 0}$ depend on three local characteristics, which are a flow $\phi=\{\phi_t\}_{t\in \mathbb{R}}$, a nonnegative jump rate function $\vartheta$, and a stochastic transition kernel $\jump$, specifying the post-jump distribution. Starting from $x$ the process $X(t)$ follows the trajectory $\phi_t (x)$ until the first jump time $\tau_1$.
Two types of jumps are possible. Either the flow $\phi_t(x)$ hits the (active) boundary of the state space $E$ in which case there is a forced jump from the boundary back to the set $E$ or a jump to a point in $E$ occurs at a rate $\vartheta$  depending on the current position of the process.  The value $X(\tau_1)$ of the process at the jump time $\tau_1$ is selected according to the distribution $\jump (\phi_{\tau_1}(x), \cdot)$ and the process restarts afresh.
For general background on PDMPs we refer the reader to \cite{davis93}.
A variety of applications has generated a renewed interest in PDMPs, see
\cite{benaim12,hairer15,liliu17,cloez18,palmowski13,rudnickityran17}
and the references therein.

Let the state space be a $\sigma$-finite measure space $(E,\mathcal{E},m)$.
Suppose that the distribution of $X(0)$ is absolutely continuous with respect to the measure $m$ with density $f$. Our main objectives are to find conditions that ensure that
the distribution of $X(t)$ is absolutely continuous with respect to $m$ for all $t>0$, and characterize an evolution equation for its density.
We use the theory of substochastic semigroups on $L^1$ spaces, as in the case of PDMPs with empty active boundary in \cite{tyran09,rudnickityran17}.  Recall that a family of linear operators  on $L^1=L^1(E,m)$ is called  a \emph{substochastic semigroup} if it is a $C_0$-semigroup of positive contractions on $L^1$, see \cite{banasiakarlotti06,rudnickityran17}.

The aim of the present paper is to build a general theory of substochastic semigroups describing the evolution of densities for piecewise deterministic Markov processes. Our approach treats in a unified way a wide class of PDMPs as described in Sections~\ref{s:PDMP} and~ \ref{s:exi}.
We introduce  assumptions  on the flow $\phi$, the jump rate function
$\vartheta$ and the jump distribution $\mathcal{P}$ (Assumptions~\ref{a:nons}--\ref{a:jump}) that allow us to show that a given process
with such characteristics induces a substochastic semigroup on the space~$L^1$ (see equation \eqref{d:indss} and Theorem~\ref{t:mPg}). To identify the generator  of this semigroup  we need to rewrite the action of the process in the space $L^1$ (see Section \ref{s:gen}). We do not assume in advance that the process is nonexplosive, but if that is the case then automatically the semigroup will be \emph{stochastic} (\cite{almcmbk94}), i.e. preserving the norm of nonnegative elements from $L^1$. Although stability and ergodicity of PDMPs are developed in great generality in \cite{costadufour08},
the general  problem of existence of absolutely continuous invariant measures has not been treated at all except for specific examples, see \cite{palmowski13} for a recent account of different models where the existence is known. If we know already that the process induces a substochastic semigroup then we can use the methods presented in \cite{pichorrudnicki16,rudnickityran17} to get existence  of invariant densities.
To complete our general approach we also study in Section~\ref{s:exiinv} relationships between invariant densities of the continuous time process and of the process observed at jump times; our results correspond to the results from \cite{costa90, davis93},  but we do not  assume that the process is non-explosive and we look for absolutely continuous invariant measures.

Section~\ref{s:pert} contains our new abstract results about substochastic semigroups. Our main tool  is a new perturbation result for substochastic semigroups   presented in Section~\ref{s:ibp}. We show in Theorem~\ref{t:gmain} that given the generator of a substochastic semigroup defined on a domain containing a zero-boundary condition we can perturb both the action of the generator and its domain to obtain a substochastic semigroup generated by an extension of the perturbed operator. Our generation result is of Kato-type \cite{kato54,voigt87} allowing also perturbation of boundary conditions as in Greiner~\cite{greiner}, but with unbounded positive operators.   In Section~\ref{s:charg} we also provide sufficient conditions
for the perturbed operator to be the generator, as well as for the perturbed semigroup to be stochastic. In Section~\ref{s:dens} we study
relationships between invariant densities of the perturbed semigroup and invariant densities of a positive contraction operator that will correspond  to the process observed at jump times. 

The proofs of our results from Section~\ref{s:main} are given in Section~\ref{s:proofs}.
First, we show in Section~\ref{s:deter} that Theorem~\ref{t:gmain} can be applied in the functional setting described in Section \ref{s:gen}. Next, in Section~\ref{s:proofm1}, we prove that the constructed substochastic semigroup actually corresponds to the given Markov process.   Section~\ref{s:inv} contains proofs of results from Section~\ref{s:exiinv}.
In Section \ref{s:examples} applications of our results are presented.  The general setting of Davis \cite{davis84,davis93} is treated in Section~\ref{s:ss}.   As a class of particular examples we treat  kinetic equations with conservative boundary conditions in Section \ref{s:kinetic}  providing probabilistic interpretation of these equations. Finally, Section \ref{s:cell} contains an application to a two-phase cell cycle model \cite{rudnickityran15}. Some auxiliary results concerning substochastic semigroups induced by flows are given in \ref{s:appen}.

\section{Main results}\label{s:main}
Let us now specify our general setting and state our main results.

\subsection{Preliminaries}\label{s:PDMP}

We consider a separable metric space $\gstate$ and a \emph{flow} $\phi=\{\phi_t\}_{t\ge 0}$ on $\gstate$,  i.e. a continuous  mapping $\phi\colon \mathbb{R}\times \gstate \to  \gstate$, $(t,x)\mapsto \phi_t(x)$, such that
\begin{equation}\label{e:ds}
\phi_0(x)=x,\quad \phi_s(\phi_t(x))=\phi_{t+s}(x)
\end{equation}
for all $t,s\in \mathbb{R}$ and all $x\in \gstate$.
Let $E^0\subset \gstate$ be a Borel set. We introduce the \emph{outgoing} boundary $\pbound$  and the \emph{incoming} boundary $\nbound$ which are points through which the flow can leave the set $E^0$  and enter the set $E^0$, respectively, given by
\begin{equation}\label{d:activep}
\pbound=\{z\in \overline{E}^0\setminus E^0: z=\phi_{ t}( x)  \text{ for some }x\in E^0, t>0, \text{ and } \phi_{ s}(x)\in E^0, s\in [0,t)\}
\end{equation}
and
\begin{equation}\label{d:activen}
\nbound=\{z\in \overline{E}^0\setminus E^0: z=\phi_{ -t}( x)  \text{ for some }x\in E^0, t>0, \text{ and } \phi_{- s}(x)\in E^0, s\in [0,t)\}.
\end{equation}
We define the \emph{hitting time} of the boundaries $\bound$ by
\begin{equation}\label{d:exit}
\life(x)=\inf\{t>0:\phi_{ t}(x)\in \pbound \}\quad \text{and}\quad \nlife(x)=\inf\{t>0:\phi_{-t}(x)\in \nbound \},\quad x\in E^0,
\end{equation}
with the convention that $\inf\emptyset=\infty$. We set $\lifepm(x)=0$ for $x\in \bound $ and we extend formula \eqref{d:exit} to  points from the boundaries $\boundn$.

The state space of a PDMP $X=\{X(t)\}_{t\ge 0}$ is taken to be  the set $E=E^0\cup \nbound\setminus (\nbound\cap \pbound)$. We consider $E$ with its Borel $\sigma$-algebra $\mathcal{E}=\mathcal{B}(E)$.
We assume that there is a \emph{jump rate function} $\vartheta\colon E\to \mathbb{R}_+$ which is a measurable function such that for each  $x\in E$
the function $r\mapsto \vartheta(\phi_r(x))$ is integrable on $[0,\varepsilon(x))$ for some $\varepsilon(x)>0$.  We consider also a \emph{jump distribution} $\jump\colon (E\cup \pbound )\times \mathcal{B}(E)\to[0,1]$ which is a transition probability, i.e.
    for each set $\F\in \mathcal{B}(E)$ the function $x\mapsto \jump(x,\F)$ is measurable and for each $x\in E\cup \pbound $ the function $\F\mapsto \jump(x,\F)$ is a probability measure.
We call  the triplet $(\phi,\vartheta,\jump )$ the \emph{characteristics of the process}.

We briefly recall from \cite{davis84,rudnickityran17} the construction of the PDMP with characteristics $(\phi,\vartheta,\jump )$.
For each $x\in E$ we define
\begin{equation}\label{d:Phi}
\Phi_x(t)=\mathbf{1}_{[0,\life(x))}(t)\exp\left\{-\int_0^t \vartheta(\phi_r(x))dr\right\},  \quad t\ge 0.
\end{equation}
Note that the function  $t\mapsto 1-\Phi_x(t)$ is the distribution function of a non-negative finite random variable, provided that $\Phi_x(\infty^{-}):=\lim_{t\to \infty}\Phi_x(t)= 0$. Otherwise, we extend $\Phi_x$ to $[0,\infty]$ by setting $\Phi_x(\infty)=1-\Phi_x(\infty^{-})$.
We also extend the state space $E$ to $\exstate=E\cup \{\Delta\}$ where $\Delta$ is a fixed state  outside $E$ representing a 'dead' state for the process and being an isolated point.
For each $x\in E$,
let $\jump(x,\{\Delta\})=0$ and $\phi_t(x)=\Delta$ if $t=\infty$. We also set $\phi_t(\Delta)=\Delta$ for all $t\ge 0$, $\jump(\Delta,\{\Delta\})=1$, and $\Phi_\Delta(t)=1$ for all $t\ge 0$.

Let $\tau_0=\sigma_0=0$ and let $X(0)=X_0$ be an $\exstate$-valued random variable on a probability space $(\Omega,\mathcal{F},\mathbb{P})$.
For each $n\ge 1$ we can choose a $[0,\infty]$-valued random variable $\sigma_n$
satisfying
\[
\mathbb{P}(\sigma_n> t|X_{n-1}=x)=\Phi_{x}(t), \quad t\ge 0.
\]
We define the $n$th \emph{jump time} by
\[
\tau_n=\tau_{n-1}+\sigma_n
\]
and we set
\begin{equation*}
X(t)=\left\{
  \begin{array}{ll}
  \phi_{t-\tau_{n-1}}(X_{n-1}) & \text{ for } \tau_{n-1}\le t<\tau_{n},\\
X_n &\text{ for } t=\tau_n,
\end{array}\right.
\end{equation*}
where the $n$th \emph{post-jump location} $X_n$ is a  random variable such that
\[
\mathbb{P}(X_n\in \F|X(\tau_{n}^{-})=x)=\jump (x,\F),\quad x\in \exstate \cup \pbound ,
\]
and $X(\tau_{n}^{-})=\lim_{t\uparrow \tau_n}X(t)$. Thus, the trajectory of the process
is defined for all $t<\tau_\infty:=\lim_{n\to\infty}\tau_n$ and $\tau_\infty$ is called the \emph{explosion time}. To define the
process for all times, we set $ X(t) =\Delta$ for $t\ge \tau_\infty$.   The process $X=\{X(t)\}_{t\ge 0}$ is  called the \emph{minimal} PDMP corresponding to $(\phi,\vartheta,\jump)$. It has right continuous sample paths, by
construction, and it is a strong Markov process. The process $X$ is said to be  \emph{non-explosive} if $\mathbb{P}_x(\tau_\infty=\infty)=1$ for  every $x\in E$.

We denote by $\mathbb{P}_x$  the distribution of the process $\{X(t)\}_{t\ge 0}$ starting at $X(0)=x$ and by $\mathbb{E}_x$ the expectation operator with respect to~$\mathbb{P}_x$. The probability transition function of the process $X$ is given by
\[
P(t,x,\F)=\mathbb{P}_x(X(t)\in \F)=\mathbb{P}_x(X(t)\in \F,t<\tau_\infty)+\mathbb{P}_x(X(t)\in \F,t\ge \tau_\infty),
\]
where $\tau_\infty$ is the explosion time.
Thus, we have $P(t,x,\F )=\mathbb{P}_x(X(t)\in \F ,t<\tau_\infty)$ for all $x\in E$ and $\F \in \mathcal{B}(E)$.
Given a $\sigma$-finite measure $m$ on the measurable space $(E,\mathcal{B}(E))$ we denote by $L^1(E,m)$ the space of integrable functions on $(E,\mathcal{B}(E),m)$.
We say that the minimal process $X=\{X(t)\}_{t\ge 0}$ \emph{induces a substochastic semigroup} $\{P(t)\}_{t\ge 0}$ on $L^1(E,m)$ if
\begin{equation}\label{d:indss}
\int_{\F }P(t)f(x)m(dx)=\int_{E}\mathbb{P}_x(X(t)\in \F ,t<\tau_{\infty})f(x)m(dx)
\end{equation}
for all $f\in L^1(E,m)$, $\F \in \mathcal{B}(E)$, $t>0$.
Suppose that the process induces a substochastic semigroup.
Then if the distribution of $X(0)$ is absolutely continuous with respect to $m$ with a Radon-Nikodym derivative $f$, called the density of  $X(0)$,  then the distribution of $X(t)$ in $E$ is absolutely continuous with respect to $m$ and its Radon-Nikodym derivative is $P(t)f$.
Since $X(t)\in E$ for $t<\tau_{\infty}$,  it follows from \eqref{d:indss} that
\[
\int_{E }P(t)f(x)m(dx)=\int_{E}\mathbb{P}_x(t<\tau_{\infty})f(x)m(dx)
\]
for all  $f\in L^1(E,m)$, $t>0$. Thus we see that $\|P(t)f\|=\|f\|$ for $f\ge 0$, $t>0$, if and only if
\[
\int_{E}(1-\mathbb{P}_x(t<\tau_{\infty}))f(x)m(dx)=0.
\]
This implies that the induced semigroup is stochastic if and only if the minimal process is $m$-a.e. non-explosive, i.e. $\mathbb{P}_x(\tau_\infty=\infty)=1$ for $m$ almost every $x\in E$.  Hence, if the process induces a stochastic semigroup and if $f$ is the density  of $X(0)$, then $P(t)f$ is the density of $X(t)$, by \eqref{d:indss}.

We conclude this section by recalling some notions from the theory of operators and semigroups on $L^1$ spaces for readers convenience. Let $(E,\mathcal{E},m)$ be a $\sigma$-finite
measure space and $L^1=L^1(E,m)$. A linear operator $P\colon L^1\to L^1$ is called \emph{substochastic} (\emph{stochastic}) if $P$ is a
positive contraction, i.e., $Pf\ge 0$  and $\|Pf\|\le \|f\|$ ($\|Pf\|= \|f\|$) for all nonnegative $f\in L^1$.
A family of substochastic (stochastic) operators $\{P(t)\}_{t\ge 0}$ on $L^1$  which is a
$C_0$-\emph{semigroup}, i.e.,
\begin{enumerate}[\upshape (1)]
\item $P(0)=I$ (the identity operator) and $P(t+s)=P(t)P(s)$ for every $s,t\ge 0$,
\item for each $f\in L^1$ the mapping $t\mapsto P(t)f$ is continuous,
\end{enumerate}
is called a \emph{substochastic (stochastic) semigroup}. The infinitesimal \emph{generator} of
a substochastic semigroup $\{P(t)\}_{t\ge0}$ is by definition the operator $G$ with domain $\mathcal{D}(G)\subset L^1$ defined as
\[
\begin{split}
\mathcal{D}(G)&=\{f\in L^1: \lim_{t\downarrow 0}\frac{1}{t}(P(t)f-f) \text{ exists} \},\\
Gf&=\lim_{t\downarrow 0}\frac{1}{t}(S(t)f-f),\quad f\in \mathcal{D}(G).
\end{split}
\]
A nonnegative $f_*$ with norm $1$ is said to be
an \emph{invariant density} for the semigroup $\{P(t)\}_{t\ge 0}$ if for each $t>0$ it is invariant for the operator $P(t)$, i.e. $P(t)f_*=f_*$.

Given a linear operator $(G,\mathcal{D}(G))$ on $L^1$
we recall that if for some real $\lambda$ the operator
$\lambda-G:=\lambda I-G$ is one-to-one, onto, and $(\lambda -G)^{-1}$ is a bounded linear operator, then $\lambda$ is said to belong to the \emph{resolvent set} $\rho(G)$ and $R(\lambda,G):=(\lambda  -G)^{-1}$ is called the \emph{resolvent
of $G$ at $\lambda$}. Following~\cite{arendt87} the operator $G$ is called \emph{resolvent positive} if there exists $\omega\in \mathbb{R}$ such that $(\omega,\infty)\subseteq \rho(G)$ and the operator $R(\lambda,G)$ is  positive  for all $\lambda>\omega$.
In particular, generators of substochastic semigroups are resolvent positive and the Hille-Yosida theorem implies the following result (see e.g. \cite[Theorem 4.4]{rudnickityran17}): A linear operator  $(G,\mathcal{D}(G))$ is the generator of a substochastic semigroup on $L^1 $ if and only if
$\mathcal{D}(G)$ is dense in $L^1$, the operator $G$ is resolvent positive, and
\begin{equation}\label{e:subsem}
\int_E Gf \,dm \le 0\quad \text{for all }f\in \mathcal{D}(G), f\ge 0.
\end{equation}
Moreover, equality holds in \eqref{e:subsem} if and only if  $(G,\mathcal{D}(G))$  generates a stochastic semigroup.

We provide sufficient conditions for the existence of a substochastic semigroup on $L^1(E,m)$ induced by the given PDMP in Section~\ref{s:exi} and we identify its infinitesimal generator in Section~\ref{s:gen}, where we are also interested in whether the induced semigroup is stochastic.

\subsection{Existence of induced substochastic semigroups}\label{s:exi}

In this section we impose general assumptions on the characteristics $(\phi,\vartheta,\jump )$ of the minimal process $X=\{X(t)\}_{t\ge 0}$ with values in $E$ as described in Section~\ref{s:PDMP} so that $X$ induces a substochastic semigroup.

We start with the properties of the flow $\phi=\{\phi_t\}_{t\in \mathbb{R}}$. We will require that the flow itself induces a stochastic semigroup by assuming that we can choose a measure $m$ on $(\gstate,\mathcal{B}(\gstate))$ in such a way that if the distribution of $X_0$ is absolutely continuous with respect to $m$, then the distribution of $\phi_t(X_0)$ is absolutely continuous with respect to $m$ for all $t$. Thus, we impose the following general assumption on the flow.

 \begin{assumption}\label{a:nons}
There exists a \emph{measurable cocycle} $\{J_t\}_{t\in \mathbb{R}}$ of $\phi$ on $\gstate$, i.e. a family of Borel measurable nonnegative functions  satisfying the following conditions
\begin{equation}\label{d:jacob}
J_0(x)=1,\quad J_{t+s}(x)=J_{t}(\phi_{s}(x))J_{s}(x),\quad s,t\in \mathbb{R}, x\in \gstate,
\end{equation}
and there exists a $\sigma$-finite  Radon measure $m$ on the Borel $\sigma$-algebra $\mathcal{B}(\gstate)$ with $m(\partial E)=0$ such that
\begin{equation}\label{d:RND}
(m\circ \phi_t^{-1})(\F )=m(\phi_t^{-1}(\F ))=\int_{\F}  J_{-t}(x)m(dx),\quad t\in \mathbb{R}, \F \in \mathcal{B}(\gstate).
\end{equation}
\end{assumption}

\begin{remark}
Condition \eqref{d:RND} implies that for each $t$ the transformation $\phi_t\colon \gstate\to \gstate$ is \emph{non-singular with respect to the measure $m$} (\cite{almcmbk94}), i.e. $m\circ \phi_t^{-1}$ is absolutely continuous with respect to $m$. Then $J_{-t}$ is the Radon-Nikodym derivative $\frac{d m\circ \phi_t^{-1}}{dm}$. Note also that condition \eqref{d:RND} together with \eqref{e:ds} implies that \eqref{d:jacob} holds for $m$-a.e. $x$. We assume in  \eqref{d:jacob} that it actually holds for all $x$.
\end{remark}

\begin{remark}\label{r:es}
Consider $\gstate=\mathbb{R}^d$ and a mapping $b\colon \mathbb{R}^d\to\mathbb{R}^d$  that is continuously differentiable with a bounded derivative. Then the ordinary differential equation $x'(t)=b(x(t))$ generates a flow $\phi\colon \mathbb{R}\times \mathbb{R}^d\to \mathbb{R}^d$ satisfying
\[
\frac{d}{dt}\phi_t(x)=b(\phi_t(x)), \quad  x\in \mathbb{R}^d, t\in \mathbb{R}.
\]
If we take as $m$ the Lebesgue measure on $\mathbb{R}^d$ then  $J_t(x)$  is the absolute value of the determinant of the derivative of the mapping $x\mapsto \phi_t(x)$, by the change of variables formula.
By Liouville's theorem, it is also given by
\begin{equation}\label{d:Liou}
J_{t}(x)
=\exp\left\{\int_0^{t} a(\phi_{r}(x))dr\right\},\quad  t\in \mathbb{R},
\end{equation}
where $a$ is the divergence of  $b$.
In a general situation, the measure $m$ might  be a product of a Lebesgue measure and a counting measure and it is hard to formulate general condition under which  Assumption \ref{a:nons} holds.
\end{remark}

As concern the jump rate function $\vartheta$ we require that the first jump time $\tau_1$ has a distribution that is absolutely continuous with respect to the Lebesgue measure on $\mathbb{R}_+$. Thus,   we assume the following condition.
\begin{assumption}\label{a:varp}
For each $x$  the function $\mathbb{R}\ni t\mapsto \int_0^t\vartheta(\phi_{r}(x))dr$ is absolutely continuous, where we
extend $\vartheta$ form $E$ to $\gstate$ by setting $\vartheta(x)=0$ for $x\not\in E$.
\end{assumption}

Next, we describe integration along the flow $\{\phi_t\}_{t\in \mathbb{R}}$.
We need to consider ``natural'' measures on the incoming $\nbound $ and the outgoing $\pbound $ parts of the boundary of $E$ that will allow us to transfer integrals over $E$ into integrals over the boundaries $\bound $.
Suppose that Assumption~\ref{a:nons} holds.
Following \cite{arlottibanasiaklods07} let
\begin{equation}\label{d:spaceEpm}
E_{\pm}=\{x\in E: 0<\lifepm(x)<\infty\},
\end{equation}
where $\life(x)$ and $\nlife(x)$ are as in \eqref{d:exit}. The properties of the flow imply that the functions $\life$ and $\nlife$ are Borel measurable and the sets
\begin{equation*}
W_{\pm}=\{(s,z): 0<s<\lifemp(z),\; z\in \bound \}
\end{equation*}
are Borel subsets of $\mathbb{R}\times \gstate$. It is easily seen that the functions $w_{\pm}\colon E_{\pm}\to W_{\pm}$ defined by
\begin{equation}\label{d:Wpm}
w_{+}(x)=(\life(x),\phi_{\life(x)}(x))\quad \text{and}\quad  w_{-}(x)=(\nlife(x),\phi_{-\nlife(x)}(x))
\end{equation}
are Borel measurable and invertible. Now, if $f$ is nonnegative and Borel measurable, then making the change of variables leads to
\begin{equation}\label{e:eplm}
\int_{E_{\pm}}f(x)m(dx)=\int_{W_{\pm}}f(\phi_{\mp s}(z))m\circ w_{\pm}^{-1}(ds,dz),
\end{equation}
where $m\circ w_{\pm}^{-1}(\F )=m(w_{\pm}^{-1}(\F ))$ for all Borel subsets $\F $ of $W_{\pm}$. We impose the following.

\begin{assumption}\label{a:dive}
There exist  finite Borel measures $m^{\pm}$ on $\bound $ such that the measures $m\circ w_{\pm}^{-1}$ can be represented by
\begin{equation}\label{e:mpm}
m\circ w_{\pm}^{-1}(\F )=\int_{\F } J_{\mp s}(z)\,ds\, m^{\pm}(dz),\quad \F \in \mathcal{B}(W_{\pm}),
\end{equation}
where $w_{\pm}$ are as in \eqref{d:Wpm} and $J_{\mp s}$ satisfy \eqref{d:jacob}.
\end{assumption}
\begin{remark}
Note that if Assumption~\ref{a:dive} holds true then it follows from \eqref{e:eplm} and \eqref{e:mpm} that, for any nonnegative and Borel measurable $f$, we have
\begin{equation}\label{e:eplus}
 \int_{E_+}f(x)\,m(dx)=\int_{\pbound }\int_{0}^{\nlife(z)} f(\phi_{-s}(z))J_{-s}(z)\,ds\, m^{+}(dz)
\end{equation}
and
\begin{equation}\label{e:eminus}
 \int_{E_{-}}f(x)\,m(dx)=\int_{\nbound }\int_{0}^{\life(z)} f(\phi_s(z))J_s(z)\,ds\, m^{-}(dz),
\end{equation}
where $E_{\pm}$ are as in \eqref{d:spaceEpm}.
Thus Assumption~\ref{a:dive} allows us to compute integrals over $E$ via integration along the flow coming from the boundary. Formula \eqref{e:eplus}  serves here as the change of variables formula in which each point $x\in E$ with $\life(x)<\infty$ can be represented  by  $x=\phi_{-s}(z)$ for some $z\in \pbound$ and $s<\nlife(z)$. Similarly in \eqref{e:eminus}, each point $x\in E$ with $\nlife(x)<\infty$  is given by  $x=\phi_{s}(z)$ for some $z\in \nbound$ and $s<\life(z)$.
\end{remark}

\begin{remark}
Note that if there exists a bounded Borel measurable function $a$ such that $J_t$ is given by \eqref{d:Liou},
then Assumption \ref{a:dive} holds true, see e.g. \cite[Proposition 3.11]{arlottibanasiaklods07}.
In particular, if $E^0$ is an open subset of $\mathbb{R}^d$ with a sufficiently regular boundary, $m$ is the Lebesgue measure and the flow is generated by $x'(t)=b(x(t))$ as in Remark \ref{r:es},  then $\bound=\{z\in \partial E^0: \pm \langle b(z),n(z)\rangle >0\}$ and
the measures $m^{\pm}$   are given by
\[
dm^{\pm}(z)=\pm\langle b(z),n(z)\rangle d\sigma(z),
\]
where $n(z)$ is the   outward normal at $z\in \partial E^0$ and $\sigma$ is the surface Lebesgue measure on $\partial E^0$.
\end{remark}

Finally, given the measures $m^{\pm}$  on $\mathcal{B}(\bound )$ as in Assumption~\ref{a:dive} the jump distribution $\mathcal{P}$  is assumed to be non-singular in the following sense. .

\begin{assumption} \label{a:jump}
There exist two positive linear operators $P_{0}\colon L^1(E,m)\times L^1(\pbound ,m^{+})\to  L^1(E,m)$ and   $P_{\partial}\colon L^1(E,m)\times L^1(\pbound ,m^{+})\to L^1(\nbound ,m^{-})$ such that, for every $\F \in \mathcal{B}(E)$,
\begin{multline}\label{e:jump}
\int_{E} \jump (x,\F )f(x)m(dx)+\int_{\pbound } \jump (x,\F )f_{\partial^+}(x)m^{+}(dx)\\ =
\int_{\F \cap E^0} P_{0}(f,f_{\partial^+})(x)m(dx)+ \int_{\F \cap \nbound } P_{\partial}(f,f_{\partial^+})(x)m^{-}(dx),
\end{multline}
for all nonnegative $f\in L^1(E,m)$ and  $f_{\partial^+}\in L^1(\pbound ,m^{+})$.
\end{assumption}

Note that in equation \eqref{e:jump}  the action of the transition kernel $\jump$ is divided into separate parts:  random jumps from $E\subseteq E^0\cup \nbound$ and  forced jumps from the boundary $\pbound$. This post-jump locations in  the set $E^0$ and in the boundary  $\nbound$ are assumed to be  absolutely continuous with respect to $m$ and $m^{-}$.
The operator $P_{0}$  is connected with jumps from the set $E\cup \pbound $  to the inside $E^0$ of $E$, while the operator $P_{\partial}$ is connected with  jumps from the set $E\cup\pbound $ to the boundary $\nbound $.

With these notations and assumptions we obtain one of the main results of the paper.
\begin{theorem}\label{t:mPg} Suppose that Assumptions~\ref{a:nons}--\ref{a:jump} hold true. Then the minimal process $\{X(t)\}_{t\ge 0}$ with characteristics $(\phi,\vartheta,\jump)$ induces a substochastic  semigroup  $\{P(t)\}_{t\ge 0}$ on $L^1(E,m)$.
\end{theorem}

The proof of Theorem \ref{t:mPg} will be given in Section \ref{s:proofm1}.
The semigroup
$\{P(t)\}_{t\ge 0}$ from Theorem~\ref{t:mPg} will be referred to as the \emph{substochastic semigroup corresponding to} $(\phi,\vartheta,\jump)$.

\subsection{Generator of the induced semigroup}\label{s:gen}
Let $(\phi,\vartheta,\mathcal{P})$ be the characteristics of the minimal process $\{X(t)\}_{t\ge 0}$ such  that Assumptions~\ref{a:nons}--\ref{a:jump}  hold true. In this section we turn to the characterization of the generator of the substochastic semigroup corresponding to $(\phi,\vartheta,\jump)$. To identify the generator we need to introduce some additional notations.

In the study of the deterministic part of the process we use the approach of~\cite{arlottibanasiaklods07,arlottibanasiaklods09}.
As in~\cite{arlottibanasiaklods09} we define the space of \emph{test functions} $\mathfrak{N}$ as follows. Let $\mathfrak{N}$ be the set of all measurable and bounded functions $\psi  \colon E\to \mathbb{R}$ with compact support in $E^0$ and such that for any $x\in E$ the function
\[
(-\nlife(x),\life(x))\ni t\mapsto \psi  (\phi_t(x))
\]
is continuously differentiable with bounded and measurable derivative at $t=0$, i.e. the mapping
\[
 x\mapsto \frac{d(\psi  \circ \phi_t)}{dt}\Big|_{t=0}(x)
\]
is bounded and measurable.
We define  the \emph{maximal transport operator} $\Tm$ on a set $\mathcal{D}_{\max}\subseteq L^1(E,m)$ as follows.  We say that
$ f\in \mathcal{D}_{\max}$ if there exists $g\in L^1(E,m)$ such that
\begin{equation}\label{d:Tm}
\int_{E} g(x)\psi  (x)m(dx)=\int_{E} f(x)\frac{d(\psi  \circ \phi_t)}{dt}\Big|_{t=0}(x)m(dx)
\end{equation}
for all $\psi  \in \mathfrak{N}$ and we set $\Tm f:=g$.

Let $m^{\pm}$ be  the measures   on $\mathcal{B}(\bound )$ as in Assumption~\ref{a:dive}. Given  $f\in L^1(E,m)$ we define its traces $\gamma^{\pm}f$ on the boundaries $\bound $ by the
the pointwise limits
\begin{equation}\label{d:traces}
\gamma^{\pm}f(z)=\lim_{s\to 0^{+}}f(\phi_{\mp s}(z))J_{\mp s}(z) 
\end{equation}
provided that the limits exist for $m^{\pm}$-a.e. $z\in \bound $. If $\bound =\emptyset$ then we set $\gamma^{\pm}= 0$.  It can be shown that $\gamma^{\pm}f$ exist for $f\in \mathcal{D}_{\max}$ (see~\ref{s:flow} and~\cite[Section 3]{arlottibanasiaklods09}).  We write
\[
\mathcal{D}(\gamma^{\pm})=\{f\in L^1(E,m):\gamma^{\pm}f\in L^1(\bound ,m^{\pm})\}.
\]
Note that the traces $\gamma^{\pm}\colon \mathcal{D}(\gamma^{\pm})\to L^1(\bound ,m^{\pm})$ are linear positive operators.
The following result corresponds to   Green's identity as in \cite[Proposition 4.6]{arlottibanasiaklods07} and its proof is given in \ref{s:appen}. Formula \eqref{e:Afin0} explains the interplay between the transport operator, the boundary measures and the traces, giving conservation of mass.

\begin{theorem}\label{l:Green0} Suppose that Assumptions~\ref{a:nons} and \ref{a:dive} hold. Let $(\Tm,\mathcal{D}_{\max})$ be the maximal transport operator  as in \eqref{d:Tm}.
If $f\in \mathcal{D}_{\max}$ is such that  $\gamma^{-}f\in L^1(\nbound ,m^-)$ then  $\gamma^+f\in L^1(\pbound ,m^+)$
 and
\begin{equation}\label{e:Afin0}
\int_{E} \Tm f(x)m(dx)=\int_{\nbound }\gamma^{-}f(x)m^{-}(dx)-\int_{\pbound }\gamma^{+}f(x)m^{+}(dx).
\end{equation}
\end{theorem}

We now define the operator $A\colon \mathcal{D}\to L^1(E,m)$ by
\begin{equation}\label{d:operatorA}
Af =\Tm f -\vartheta f,\quad f\in \mathcal{D},
\end{equation}
where the transport operator $\Tm$ is as in \eqref{d:Tm} and
\begin{equation}\label{d:domainopA}
\mathcal{D}=\{f\in \mathcal{D}_{\max}: \gamma^{-}f \in L^1(\nbound ,m^{-}), \vartheta f\in L^1(E,m)\}.
\end{equation}
Note that $\mathcal{D}\subset \mathcal{D}(\gamma^{+})$, by Theorem~\ref{l:Green0}.
The next result implies that a restriction of the operator $A$ is the generator of a substochastic semigroup. It  extends  the result of \cite{arlottibanasiaklods07} and its proof is given in \ref{s:appen}.
\begin{theorem}\label{t:gs0t} Suppose that Assumptions~\ref{a:nons}--\ref{a:dive} hold. Let $(A,\mathcal{D})$ be as in \eqref{d:operatorA}--\eqref{d:domainopA}.
Then the operator $(A_0,\mathcal{D}(A_0))$, defined as the restriction of the operator $(A,\mathcal{D})$
\begin{equation*}
A_0f=Af, \quad f\in\mathcal{D}(A_0)=\{f\in \mathcal{D}: \gamma^{-}f=0\},
\end{equation*}
is the generator of a substochastic semigroup $\{S(t)\}_{t\ge 0}$ on $L^1(E,m)$
given by
\begin{equation}\label{d:ss0}
S(t)f(x)=\mathbf{1}_{E}(\phi_{-t}(x))f(\phi_{-t}(x))J_{-t}(x)e^{-\int_0^t \vartheta(\phi_{-r}(x))dr}
\end{equation}
for $t>0$, $x\in E$, $f\in L^1(E,m)$. Moreover,
\begin{equation}\label{eq:T0St}
\int_E \psi (x) S(t)f(x)\,m(dx)=\int_E e^{-\int_{0}^{t}\vartheta(\phi_r(x))dr}\mathbf{1}_{[0,\life(x))}(t) \psi (\phi_t(x))f(x)\,m(dx)
\end{equation}%
for all $t\ge 0$, $f\in  L^1(E,m)$, $f\ge 0$, and all nonnegative Borel measurable $\psi $.
\end{theorem}

Our second main result provides a functional-analytic description of the minimal process.

\begin{theorem}\label{th:ptsmal} Suppose that $(\phi,\vartheta,\jump)$ satisfy Assumptions~\ref{a:nons}--\ref{a:jump}.  Let $(A,\mathcal{D})$ be defined by~\eqref{d:operatorA}--\eqref{d:domainopA} and let
$B\colon \mathcal{D}\to L^1(E,m)$, $\Psi\colon \mathcal{D}\to L^1(\nbound ,m^{-})$ be given  by
\begin{gather}\label{d:part}
Bf=P_{0}(\vartheta f, \gamma^+  f),\quad \Psi f=P_{\partial} (\vartheta f, \gamma^{+}f), \quad f\in \mathcal{D},
\end{gather}
where $P_0$, $P_{\partial}$ satisfy \eqref{e:jump}.
Then
the generator $(G,\mathcal{D}(G))$ of the semigroup $\{P(t)\}_{t\ge 0}$ corresponding to $(\phi,\vartheta,\jump)$  is an extension of the operator $(A_{\Psi}+B,\mathcal{D}(A_{\Psi}))$, i.e.
\[
Gf=A_{\Psi}f+Bf,\quad A_{\Psi}f=Af,\quad  f\in \mathcal{D}(A_{\Psi})=\{f\in \mathcal{D}:\gamma^{-}f=\Psi f\}.
\]
Moreover, if $\mathcal{D}(G)=\mathcal{D}(A_{\Psi})$ then $\{P(t)\}_{t\ge 0}$ is stochastic.
\end{theorem}

The proof of Theorem \ref{th:ptsmal} will be given in Section \ref{s:proofm1}.  The idea of the proofs of  Theorems \ref{t:mPg} and  \ref{th:ptsmal} is the following. We see that $A_{\Psi}+B$ is a perturbation of the generator   $(A_0,\mathcal{D}(A_0))$ of a substochastic semigroup  from Theorem~\ref{t:gs0t} with $B$ changing the action of the operator $A_0$ and $\Psi$ changing its domain.
In particular, if $\Psi$ were a bounded operator  then $A_{\Psi}$ is the generator of a $C_0$-semigroup by Greiner's perturbation theorem \cite{greiner} and the existence of the semigroup $\{P(t)\}_{t\ge 0}$ with generator as described in Theorem~\ref{th:ptsmal} could be deduced form Kato--Voigt perturbation theorem \cite{voigt87,banasiakarlotti06} by showing that $A_{\Psi}$ generates a substochastic semigroup.  However, in general, the operator $\Psi$ might be  unbounded or $A_{\Psi}$ might not the generator.  To take these into account  we provide a new perturbation theorem for substochastic semigroups in Section~\ref{s:pert}
and we show in Section~\ref{s:proofs} that it can be applied in the setting of Theorem~\ref{th:ptsmal} implying the existence of the induced semigroup $\{P(t)\}_{t\ge 0}$.

Finally, consider the initial-boundary value problem
\begin{align}\label{e:ACP1}
\frac{\partial u}{\partial t}&=\Tm u-\vartheta u+P_{0}(\vartheta u,\gamma^{+} u),\\
\gamma^{-}u&=P_{\partial}(\vartheta u,\gamma^{+}u),\quad u(0)=f,\label{e:ACP2}
\end{align}
where $\Tm$ is the transport operator and  $P_0$, $P_{\partial} $ satisfy \eqref{e:jump}. 
Recall that  the Cauchy problem \eqref{e:ACP1}--\eqref{e:ACP2} is well posed if and only if  the operator $(A_{\Psi}+B,\mathcal{D}(A_{\Psi}))$ is the generator of a $C_0$-semigroup. Theorem~\ref{th:ptsmal}  shows only that an extension $G$ of the operator $(A_{\Psi}+B,\mathcal{D}(A_{\Psi}))$ is the generator. However, if $\mathcal{D}(G)=\mathcal{D}(A_\Psi)$ and $f$ is a density of $X(0)$, then $u(t)=P(t)f$ is the density of $X(t)$, $t> 0$, and $u$ satisfies \eqref{e:ACP1}--\eqref{e:ACP2}, so that this equation  can be called the Fokker--Planck equation for our Markov process.
Thus,  we need to impose additional constraints to conclude that $\mathcal{D}(G)=\mathcal{D}(A_\Psi)$. One set of such conditions is given in the next result, yet another is provided in Section~\ref{s:charg}.

\begin{corollary}\label{c:sto}
In addition to Assumptions~\ref{a:nons}--\ref{a:jump} suppose that $\vartheta$ is bounded and that either $\pbound =\emptyset$ or $\jump(z,\nbound )=1$, $z\in \pbound $, with $\inf\{\life(z):z\in \nbound \}>0$.
Then the semigroup $\{P(t)\}_{t\ge 0}$ corresponding to $(\phi,\vartheta,\jump)$ is stochastic and its generator is the operator $(A_{\Psi}+B,\mathcal{D}(A_{\Psi}))$.
\end{corollary}
The proof of Corollary \ref{c:sto} will be given in Section \ref{s:proofm1}.
Note that the condition $\pbound =\emptyset$ implies that the operators $P_0\colon L^1(E,m)\times L^1(\pbound,m^+)\to L^1(E,m)$ and $P_{\partial}\colon L^1(E,m)\times L^1(\pbound,m^+)\to L^1(\nbound,m^{-})$ are defined on $L^1(E,m)$, while $\jump(z,\nbound )=1$, $z\in \pbound $, implies  that the operator $P_0 $ has to be defined only  on $L^1(E,m)$ and the operator $P_{\partial}$   on $L^1(\pbound,m^+)$.

\begin{remark}
Note that one of the standard assumptions in \cite{davis84, davis93} about the process $X=\{X(t)\}_{t\ge 0}$ is the following condition
\begin{equation}\label{c:finnt}
\mathbb{E}_x (N_t)<\infty, \quad x\in E, t>0,\quad \text{where}\quad N_t=\sup\{n:\tau_n\le t\}.
\end{equation}
It implies that in every finite time interval there is a finite number of jump times $\tau_n$ and that
$\mathbb{P}_x(\tau_n\to \infty)=1$ for all $x\in E$.  In particular, the process $X$ is then non-explosive and if  Assumptions~\ref{a:nons}--\ref{a:jump} hold true then the  induced semigroup $\{P(t)\}_{t\ge 0}$ is stochastic.

Assuming~\eqref{c:finnt} it follows from \cite{davis84, davis93}  that if we define
\[
v(t,x)=\mathbb{E}_x(\psi (X(t))), \quad x\in E, t\ge 0,
\]
then, for any sufficiently smooth bounded function $\psi \colon E\to \mathbb{R}$, 
the function $v$ satisfies the following Kolmogorov equation
\begin{equation}\label{d:eveqdu}
\frac{\partial v}{\partial t}=\mathfrak{X} v-\vartheta v+q\jump v, \quad \gamma^{+}v=\jump v,
\end{equation}
with initial condition $v(0,x)=\psi (x)$, $x\in E$, where the operators  $\mathfrak{X}, \jump,\gamma^{+}$ are given by
\[
\mathfrak{X}\psi (x)=\frac{d(\psi  \circ \phi_t)}{dt}\Big|_{t=0}(x),\quad x\in E, \quad \jump \psi (x)=\int_E \psi (y)\jump(x,dy),\quad  x\in E\cup \pbound,
\]
and $\gamma^{+}\psi (x)=\lim_{t\to 0}\psi (\phi_{-t}(x))$, $x\in \pbound $. It follows from \eqref{d:indss} that
\[
\int_{E}\psi (x)P(t)f(x)m(dx)=\int_E v(t,x)f(x)m(dx),\quad f\in L^1(E,m).
\]
However, this duality does not show directly the differences in the boundary  conditions in
equation~\eqref{d:eveqdu} and in the Cauchy problem \eqref{e:ACP1}--\eqref{e:ACP2}.
\end{remark}

\subsection{Invariant densities for induced semigroups}
\label{s:exiinv}
Let $(\phi,\vartheta,\mathcal{P})$ be the characteristics of the minimal process $X=\{X(t)\}_{t\ge 0}$
such  that Assumptions~\ref{a:nons}--\ref{a:jump} hold true.  In this section  we study the relationships between invariant densities of the substochastic semigroup $\{P(t)\}_{t\ge 0}$ corresponding to $(\phi,\vartheta,\jump)$ and invariant densities for the process observed   at jump  times $\tau_n$, $n\ge 0$.
First, we define a linear operator $K\colon L^1(E,m)\times L^1(\nbound,m^{-})\to   L^1(E,m)\times L^1(\nbound,m^{-})$ by
\begin{equation}\label{d:operatorK}
K (f,f_{\partial})=\big(P_0(\vartheta R_0(f,f_\partial),R_0(f,f_\partial)),P_{\partial}(\vartheta R_0(f,f_\partial),R_0(f,f_\partial))\big),
\end{equation}
where
$P_0,P_\partial$ satisfy \eqref{e:jump} and
\begin{multline}\label{d:R0p}
R_0(f,f_{\partial})(x)=\int_{0}^{\nlife(x)}e^{-\int_0^t \vartheta(\phi_{-r }(x))dr} f(\phi_{-t}(x))J_{-t}(x)dt\\ +\mathbf{1}_{\{\nlife(x)<\infty\}}e^{-\int_0^{\nlife(x)} \vartheta(\phi_{-r }(x))dr} f_\partial(\phi_{-\nlife(x)}(x))J_{-\nlife(x)}(x),\quad x\in E\cup \Gamma^{+},
\end{multline}
for nonnegative $(f,f_{\partial})\in L^1(E,m)\times L^1(\nbound ,m^{-})$.
The proof of our next result will be given in Section \ref{s:inv}.
\begin{theorem}\label{t:Kkernel}
The transition kernel $\mathcal{K} (x,\cdot)$ of the discrete-time Markov process $(X(\tau_n))_{n\ge0}$ is
given by
\begin{equation*}\label{eq:Kkernel}
\mathcal{K} (x,\F )=\int_0^{\life(x)}e^{-\int_0^s \vartheta(\phi_r(x))dr} \vartheta(\phi_s(x))\jump(\phi_s(x),\F )ds + \mathbf{1}_{\{\life(x)<\infty\}}e^{-\int_0^{\life(x)} \vartheta(\phi_r(x))dr}\jump(\phi_{\life(x)}(x),\F )
\end{equation*}
for $x\in E$, $\F \in \mathcal{B}(E)$. The operator  $K=(K_0,K_{\partial})$ as defined in \eqref{d:operatorK} is substochastic on $L^1(E,m)\times L^1(\nbound,m^{-})$ and satisfies
\begin{multline*}
\int_E \mathcal{K} (x,\F )f(x)m(dx)+\int_{\nbound}\mathcal{K} (x,\F )f_{\partial}(x)m^{-}(dx)=\int_{\F\cap E^0}K_{0}(f,f_{\partial})m(dx)+\int_{\F\cap\nbound}K_{\partial}(f,f_{\partial})(x)m^{-}(dx)
\end{multline*}
for all $\F \in  \mathcal{B}(E)$, $(f,f_{\partial})\in L^1(E,m)\times L^1(\nbound ,m^{-})$. Moreover, $K $ is stochastic, if for every $x$ with $\life(x)=\infty$ we have
\begin{equation}\label{d:inflife}
\lim_{t\to \life(x)}\int_0^{t}\vartheta(\phi_r(x))dr=\infty.
\end{equation}
\end{theorem}

If $(f,f_{\partial})\in L^1(E,m)\times L^1(\nbound,m^{-})$ is nonnegative with norm $1$ and  $K (f,f_{\partial})=(f,f_{\partial})$, then $(f,f_{\partial})$ is called an \emph{invariant density for the operator} $K $.
We have the following result, corresponding to \cite[Theorem 2]{costa90} and \cite[Theorem 34.31]{davis93}.
\begin{theorem}\label{c:invsem}
Suppose that $(f_{},f_{\partial})$ is an invariant density for the operator $K $ such that
\begin{equation}\label{a:intinv}
c:=\int_E\int_0^{\life(x)}e^{-\int_0^t\vartheta(\phi_r(x))dr}dtf(x)m(dx)+\int_{\nbound}\int_0^{\life(x)}
e^{-\int_0^t\vartheta(\phi_r(x))dr}dtf_{\partial}(x)m^{-}(dx)<\infty.
\end{equation}
Let
\begin{equation}\label{d:invsemi}
\overline{f}(x)=\int_{0}^{\nlife(x)}e^{-\int_0^t \vartheta(\phi_{-r }(x))dr} f(\phi_{-t}(x))J_{-t}(x)dt +\mathbf{1}_{\{\nlife(x)<\infty\}}e^{-\int_0^{\nlife(x)} \vartheta(\phi_{-r }(x))dr} f_\partial(\phi_{-\nlife(x)}(x))J_{-\nlife(x)}(x).
\end{equation}
Then
${f}_*=c^{-1}\overline{f}$  is an invariant density for the semigroup $\{P(t)\}_{t\ge 0}$.
\end{theorem}
The proof of Theorem \ref{c:invsem} will be given in Section~\ref{s:inv}.
To relate our result to  \cite[Theorem 2]{costa90} and \cite[Theorem 34.31]{davis93} observe that if $(f_{},f_{\partial})$ is an invariant density for the operator $K $ then the probability measure $\pi$ defined~by
\begin{equation}\label{d:mus}
\pi(\F )=\int_{\F \cap E^0}f(x)m(dx)+\int_{\F \cap \nbound }f_{\partial}(x)m^{-}(dx),\quad \F \in \mathcal{B}(E),
\end{equation}
is invariant for the discrete-time process $(X(\tau_n))_{n\ge 0}$, since it satisfies
\begin{equation*}
\int_{E}\mathcal{K} (x,\F )\pi(dx)=\pi(\F ),\quad \F \in \mathcal{B}(E),
\end{equation*}
by Theorem~\ref{t:Kkernel}. In the proof of Theorem~\ref{c:invsem} we  show in fact that
\begin{equation}\label{d:norma}
\int_{\F}\overline{f} dm=\int_E\int_0^{\life(x)}e^{-\int_0^t\vartheta(\phi_r(x))dr}\mathbf{1}_{\F}(\phi_t(x))dt\,\pi(dx),\quad \F \in \mathcal{B}(E).
\end{equation}
Thus assumption \eqref{a:intinv}  is as in \cite[Theorem 2]{costa90} and \cite[Theorem 34.31]{davis93}, as well as the invariant measure for the process $\{X(t)\}_{t\ge 0}$ being of the form
\[
\mu(\F )=
\frac{\int_E\int_0^{\life(x)}e^{-\int_0^t\vartheta(\phi_r(x))dr}\mathbf{1}_{\F}(\phi_t(x))dt\,\pi(dx)}
{\int_E\int_0^{\life(x)}e^{-\int_0^t\vartheta(\phi_r(x))dr}dt\,\pi(dx)}=\int_{\F} {f}_*(x)m(dx),\quad \F \in \mathcal{B}(E).
\]
However, we additionally obtained that the measure $\mu$  is absolutely continuous with respect to $m$.

We have also the following converse result. It corresponds to \cite[Theorem 1]{costa90} and \cite[Theorem 34.21]{davis93} and its proof will be given  in Section~\ref{s:inv}.

\begin{theorem}\label{c:invconv}
Suppose that ${f}_*$ is an invariant density for the semigroup $\{P(t)\}_{t\ge 0}$ and that $K $ is stochastic.
Then
\[
0<c_*:=\int_E \vartheta(x){f}_*(x)m(dx)+\int_{\pbound }\gamma^{+} {f}_*(x)m^{+}(dx)<\infty
\]
and the operator $K $ has an invariant density
$(f_{},f_{\partial})$ given by
\begin{equation}\label{d:invdens}
f=c_{*}^{-1}P_0(\vartheta f_*,\gamma^{+} f_{*}), \quad f_{\partial}=c_{*}^{-1}P_{\partial}(\vartheta f_*,\gamma^{+} f_{*}).
\end{equation}
\end{theorem}

In particular, in the setting of Theorem \ref{c:invconv}, the invariant measure $\pi$ as defined in \eqref{d:mus} with $f$ and $f_{\partial}$ given by \eqref{d:invdens} now satisfies
\[
\pi (\F )=\frac{\int_E\jump(x,\F ) \vartheta(x) f_*(x)m(dx)+\int_{\pbound }\jump(x,\F )\gamma^{+}f_*(x)m^{+}(dx)}{\int_E \vartheta(x)f_*(x)m(dx)+\int_{\pbound }\gamma^{+} f_*(x)m^{+}(dx)},\quad \F \in \mathcal{B}(E).
\]
This formula agrees with the one in \cite[Theorem 34.21]{davis93} where in equation (34.23) the boundary measure $\sigma$ is given by $\sigma(dx)=\gamma^{+} f_*(x)m^{+}(dx)$.

\section{Perturbation theorem for substochastic semigroups}\label{s:pert}

In this section we combine the perturbation methods of Kato \cite{kato54} and Greiner~\cite{greiner} to obtain substochastic semigroups by perturbing both the generator of a substochastic semigroup as well as  boundary conditions. For the perturbation theory of operator semigroups we refer the reader to \cite[Chapter III]{engelnagel00} and \cite{banasiakarlotti06}.
A number of  perturbation results with unbounded perturbations of boundary conditions has been obtained recently in \cite{adler2014perturbation,adler2017perturbation,rhandi15}.
Our generation theorem is stated in Section~\ref{s:ibp} and it gives sufficient conditions for the existence of a substochastic semigroup with generator being an extension of the given operator.  The proof is given
by adapting the ideas of Kato~\cite{kato54} to our setting. Since our generation theorem does not give the full characterization
of the generator, we present sufficient conditions for the given operator to be the generator in Section~\ref{s:charg}.  Finally, in Section~\ref{s:dens} we extend  results from \cite[Section 3]{biedrzyckatyran}  that will be used in the sequel to prove Theorems~\ref{t:Kkernel}--\ref{c:invconv}.

\subsection{Inner and boundary perturbations}\label{s:ibp}

Let $(E,\mathcal{E},m)$ be a $\sigma$-finite
measure space and $L^1=L^1(E,m)$. 
We assume that there is a second $L^1$ space denoted by $L^1_{\partial}=L^1(E_{\partial},m_{\partial})$, where $(E_{\partial},\mathcal{E}_{\partial},m_{\partial})$ is a $\sigma$-finite measure space; it will serve here as the boundary space. Let $\mathcal{D}$ be a linear subspace of $ L^1$. We consider a linear operator $A\colon \mathcal{D}\to L^1$, called the \emph{maximal operator} in the sense that it has a sufficiently big domain, a positive operator $B\colon \mathcal{D}\to L^1$  and two linear positive operators
$\Psi_0,\Psi\colon \mathcal{D}\to L^1_{\partial}$,  called boundary operators.

We assume throughout this section that
\begin{enumerate}[(i)]
\item\label{S1}  the operator $(A_0,\mathcal{D}(A_0))$ defined by 
\begin{equation*}\label{d:dombou0}
A_0f=Af, \quad f\in\mathcal{D}(A_0)=\{f\in \mathcal{D}: \Psi_0f=0\},
\end{equation*}
is the generator of a substochastic semigroup on $L^1$;
\item\label{S2} if $\Psi_0\not\equiv 0$ then for each $\lambda>0$ the operator $\Psi_0$ restricted to the nullspace $\Ker(\lambda -A)=\{f\in \mathcal{D}:\lambda f-Af=0\}$ has a \emph{positive right inverse}, i.e. there exists a positive operator $\Psi(\lambda)\colon L^1_{\partial} \to \Ker(\lambda-A)$ such that $\Psi_0\Psi(\lambda)f_{\partial}=f_{\partial}$ for $f_{\partial}\in L^1_{\partial}$;

\item\label{S4} for each nonnegative  $f\in \mathcal{D}$ the following holds
\begin{equation}\label{eq:zero}
\int_{E}(Af+Bf)\,dm+\int_{E_{\partial}}(\Psi f-\Psi_0f) \,dm_{\partial}\le 0.
\end{equation}
\end{enumerate}

We can now formulate our generalization of Kato's and Greiner's results.
\begin{theorem}\label{t:gmain} Assume conditions \eqref{S1}--\eqref{S4}. Let the operator  $(A_{\Psi},\mathcal{D}(A_\Psi))$ be defined by
\begin{equation}
\label{d:dombou}
A_{\Psi}f=Af, \quad f\in \mathcal{D}(A_{\Psi})=\{f\in \mathcal{D}: \Psi_0 f=\Psi f\}.
\end{equation}
Then there exists a  substochastic semigroup $\{P(t)\}_{t\ge 0}$ on $L^1$ with generator $(G,\mathcal{D}(G))$
being an extension of $(A_{\Psi}+B,\mathcal{D}(A_{\Psi}))$. The resolvent operator of $G$ at $\lambda>0$ is given by
\begin{equation}\label{d:resg}
R(\lambda,G)f=\lim_{N\to \infty} \sum_{n=0}^{N}(R(\lambda,A_0)B+\Psi(\lambda)\Psi)^n R(\lambda,A_0)f,\quad f\in L^1.
\end{equation}
\end{theorem}

\begin{remark}
\begin{enumerate}[(a)]
\item
If the boundary operators are zero, i.e.  $\Psi_0=\Psi\equiv 0$,  then $A_0=A$, $\Psi(\lambda)\Psi=0$ and Theorem~\ref{t:gmain} goes back to the work of Kato~\cite{kato54}, as formulated and extended in \cite{voigt87,banasiakarlotti06}.

\item If, on the other hand, $B=0$ then Theorem~\ref{t:gmain} is a particular extension of Greiner's theorem \cite{greiner}, where it was assumed that the boundary perturbation $\Psi$ is bounded;  it can also be compared with the generation result from \cite{GMTKpositive}.

\item Note that it follows from  \cite[Lemma 1.2]{greiner} that condition \eqref{S2} holds, if $(A,\mathcal{D})$ is closed, $\Psi_0$ is onto and continuous with respect to the graph norm $\|f\|_A=\|f\|+\|Af\|$. The operators $\Psi(\lambda)$ are so called abstract Dirichlet operators \cite{adler2014perturbation,adler2017perturbation}.

\item Finally, since for $f\in  \mathcal{D}(A_{\Psi})$ we have $\Psi f-\Psi_0 f=0$, condition~\eqref{S4} implies that condition \eqref{e:subsem} holds at least for nonnegative $f\in  \mathcal{D}(A_{\Psi})$. 

\end{enumerate}
\end{remark}

Before we give the proof of Theorem \ref{t:gmain} we need to introduce some preliminary notations.
We consider the space $\mathcal{X}=L^1\times L^1_{\partial}$  with norm
\[
\|(f,f_{\partial})\|=
\int_E |f(x)|m(dx)+\int_{E_{\partial}}|f_{\partial}(x)|m_{\partial}(dx),\quad (f,f_{\partial})\in L^1\times L^1_{\partial},
\]
and we define the operators $\mathcal{A},\mathcal{B}\colon \mathcal{D}(\mathcal{A}) \to L^1\times L^1_{\partial}$ with $\mathcal{D}(\mathcal{A})=\mathcal{D} \times \{0\}$  by
\begin{equation}\label{d:AB}
\mathcal{A}(f,0)=(Af,-\Psi_{0}f)\quad \text{and}\quad \mathcal{B}(f,0)=(B f,\Psi f) \quad
\textrm{for $f\in \mathcal{D}$}.
\end{equation}
The resolvent of the operator $\mathcal{A}$ at $\lambda>0$ is given by (see e.g. \cite[Section 3.3.4]{rudnickityran17})
\begin{equation}\label{eq:RA}
R(\lambda,  \mathcal{A})(f,f_{\partial})=(R(\lambda,A_0)f+\Psi(\lambda) f_{\partial},0),\quad (f,f_{\partial})\in L^1\times L^1_{\partial}.
\end{equation}
By assumption, the operators $R(\lambda,\mathcal{A})$, $B$ and $\Psi$ are positive. Thus the operators $\mathcal{B}$ and $\mathcal{B}R(\lambda,\mathcal{A})$ are positive.
We have
\begin{equation}\label{eq:sum1}
\mathcal{B}R(\lambda,  \mathcal{A}) +\lambda R(\lambda,  \mathcal{A})=\mathcal{I} +(\mathcal{A}+\mathcal{B})R(\lambda,  \mathcal{A}),
\end{equation}
where $\mathcal{I}$ is the identity operator on  $L^1\times L^1_{\partial}$.
Since $R(\lambda,  \mathcal{A})(f,f_{\partial})\in \mathcal{D}\times \{0\}$, we use  condition~\eqref{eq:zero} to conclude that
\begin{equation*}
\|\mathcal{B}R(\lambda,  \mathcal{A})(f,f_{\partial})\|+ \|\lambda R(\lambda,  \mathcal{A})(f,f_{\partial})\|\le \|(f,f_{\partial})\| \end{equation*}
for nonnegative $f$ and $f_{\partial}$. This implies that the operators $\mathcal{B}R(\lambda,  \mathcal{A})$ and $\lambda R(\lambda,\mathcal{A})$ are positive contractions on $\mathcal{X}=L^1\times L^1_{\partial}$.
We have
\begin{equation}\label{eq:BRA}
\mathcal{B}R(\lambda,  \mathcal{A})(f,f_{\partial})=(BR(\lambda,A_0)f+B\Psi(\lambda)f_{\partial}, \Psi R(\lambda,A_0)f+\Psi\Psi(\lambda)f_{\partial})
\end{equation}
for $\lambda>0$, $(f,f_{\partial})\in L^1\times L^1_{\partial}$.

In the proof of Theorem~\ref{t:gmain} we apply the argument of Kato~\cite{kato54} to  the operator  $(\mathcal{A}+\mathcal{B},\mathcal{D}(\mathcal{A}))$  in the space $L^1\times L^1_{\partial}$.
However, the main difficulty now is that $\mathcal{D}(\mathcal{A})$ is not dense in $L^1\times L^1_{\partial}$. 
We have  $\overline{\mathcal{D}(\mathcal{A})}=L^1\times \{0\}$, since  $\mathcal{D}(A_0)\times \{0\}\subset \mathcal{D}(\mathcal{A})\subset L^1\times \{0\}$ and the domain  of the generator $A_0$ of a substochastic semigroup is dense in $L^1$.
The part of $(\mathcal{A}+\mathcal{B},\mathcal{D}(\mathcal{A}))$ in $\mathcal{X}_0=\overline{\mathcal{D}(\mathcal{A})}=L^1\times \{0\}$, denoted by $(\mathcal{A}+\mathcal{B})_{|}$  and being the restriction of $\mathcal{A}+\mathcal{B}$ to the domain
\[
\begin{split}
\mathcal{D}((\mathcal{A}+\mathcal{B})_{|})&=\{(f,f_{\partial})\in \mathcal{D}(\mathcal{A})\cap \mathcal{X}_0: (\mathcal{A}+\mathcal{B})(f,f_{\partial})\in \mathcal{X}_0\},
\end{split}
\]
can be identified with $(A_{\Psi}+B,\mathcal{D}(A_{\Psi}))$, since $\mathcal{D}((\mathcal{A}+\mathcal{B})_{|})= \mathcal{D}(A_{\Psi})\times \{0\}$
and
\[
(\mathcal{A}+\mathcal{B})_{|}(f,0)=(A_{\Psi}f+Bf,0), \quad f\in \mathcal{D}(A_{\Psi}).
\]
We make use of the following result that easily follows from \cite[Corollary II.3.21]{engelnagel00}.
\begin{lemma}\label{t:geth} Assume conditions \eqref{S1}--\eqref{S4}.
If, for each $\lambda>0$, the operator $\mathcal{I}-\mathcal{B}R(\lambda,\mathcal{A})$ is invertible with positive inverse, then
the resolvent of $\mathcal{A}+\mathcal{B}$ at $\lambda>0$ is given by
\begin{equation}\label{e:rab}
R(\lambda,\mathcal{A}+\mathcal{B})=R(\lambda,\mathcal{A})(\mathcal{I}-\mathcal{B}R(\lambda,\mathcal{A}))^{-1}
\end{equation}
and $\lambda \|R(\lambda,\mathcal{A}+\mathcal{B})\|\le 1$ for all $\lambda>0$. Moreover, the part $(\mathcal{A}+\mathcal{B})_{|}$ of the operator  $(\mathcal{A}+\mathcal{B},\mathcal{D}(\mathcal{A}))$ in $\mathcal{X}_0=\overline{\mathcal{D}(\mathcal{A})}$ is densely defined in~$\mathcal{X}_0$ and generates a $C_0$-semigroup of positive contractions on $\mathcal{X}_0$.
\end{lemma}

\begin{remark}\label{r:non}
\begin{enumerate}[(a)]
\item\label{r:nona} It follows from \cite[Lemma 1.3]{greiner} that given any $\lambda,\mu\in\rho(A_0)$ we have
\begin{equation}\label{eq:Psilam}
\Psi(\lambda)=\Psi(\mu)+(\mu-\lambda)R(\lambda,A_0)\Psi(\mu).
\end{equation}
Since
$R(\lambda,A_0)\Psi(\mu)\ge 0$, we see that $\Psi(\mu)\le \Psi(\lambda)$ for $\mu>\lambda$.
\item Since $\mathcal{B}R(\lambda,\mathcal{A})$ is a positive operator,  the operator $\mathcal{I}-\mathcal{B}R(\lambda,\mathcal{A})$ is invertible with positive inverse if and only if
the spectral radius of the operator $\mathcal{B}R(\lambda,\mathcal{A})$ is strictly smaller than 1,
or equivalently,
\begin{equation}\label{eq:ssper}
\lim_{n\to \infty}\|(\mathcal{B} R(\lambda,\mathcal{A}))^n\|=0.
\end{equation}
We have $\mathcal{B}R(\mu,\mathcal{A})\le \mathcal{B}R(\lambda,\mathcal{A})$ for $\mu>\lambda$ and we see that condition \eqref{eq:ssper} holds for all $\lambda$ sufficiently large, if it holds for one $\lambda>0$.
\item In Lemma \ref{t:geth} it is enough to assume that the operator $(\mathcal{A}+\mathcal{B},\mathcal{D}(\mathcal{A}))$ is resolvent positive, or equivalently, by \cite[Theorem 1.1]{voigt89}, that condition \eqref{eq:ssper} holds for one $\lambda>0$.
\end{enumerate}
\end{remark}

With these preparations  we can now turn to the
\begin{proof}[Proof of Theorem~\ref{t:gmain}]
For each $r\in [0,1)$ consider the operator $\mathcal{G}_r=\mathcal{A}+r\mathcal{B}$ with domain $\mathcal{D}(\mathcal{A})=\mathcal{D}\times \{0\}$.
Since $B$ and $\Psi$ are positive operators, we see that condition \eqref{eq:zero} still holds for the positive operators $rB$ and $r\Psi$, $r\in [0,1)$. We have $\|r\mathcal{B}R(\lambda,  \mathcal{A})\|\le r<1$ for $r\in [0,1)$. Thus, for each $\lambda>0$ and $r\in [0,1)$, the operator $\mathcal{I}-r\mathcal{B}R(\lambda,\mathcal{A})$ is invertible with positive inverse.
From Lemma~\ref{t:geth} it follows that
\[
R(\lambda,\mathcal{G}_r)=R(\lambda,\mathcal{A})\sum_{n=0}^\infty r^n (\mathcal{B}R(\lambda,\mathcal{A}))^n,
\]
$\|R(\lambda,\mathcal{G}_r)\|\le \lambda^{-1}$ and that the part ${\mathcal{G}_r}_{|}$ of
the operator $(\mathcal{G}_r,\mathcal{D}(\mathcal{A}))$ in $\overline{\mathcal{D}(\mathcal{A})}$ is the generator of a $C_0$-semigroup $\{\mathcal{P}_{r}(t)\}_{t\ge 0}$ of positive contractions  on $\overline{\mathcal{D}(\mathcal{A})}$.
Arguing as in \cite{kato54}
we conclude that the family of operators $\{\mathcal{P}(t)\}_{t\ge 0}$ defined by
\[
\mathcal{P}(t)(f,f_{\partial})=\lim_{r\to 1}\mathcal{P}_{r}(t)(f,f_{\partial}),\quad (f,f_{\partial})\in \overline{\mathcal{D}(\mathcal{A})}=L^1\times\{0\},
\]
is a $C_0$-semigroup of positive contractions on $\overline{\mathcal{D}(\mathcal{A})}$. Let $(\mathcal{G},\mathcal{D}(\mathcal{G}))$ be the generator of  $\{\mathcal{P}(t)\}_{t\ge 0}$ and $R(\lambda,\mathcal{G})$ be its resolvent at $\lambda>0$. We take
\[
R(\lambda,G)f=\Pi_1 R(\lambda,\mathcal{G})(f,0)\quad \text{and}\quad P(t)f=\Pi_1\mathcal{P}(t)(f,0),
\]
where $\Pi_1(f,f_{\partial})=f$.

Since $0\le R(\lambda,\mathcal{G}_r)\le R(\lambda,\mathcal{G}_{r'})$ for $r<r'$ and $\|R(\lambda,\mathcal{G}_r)\|\le \lambda^{-1}$, we see that the limit
\[
\mathcal{R}_\lambda(f,f_{\partial}) =\lim_{r\uparrow 1}R(\lambda,\mathcal{G}_r)(f,f_{\partial})
\]
exists for all  $(f,f_{\partial})\in L^1\times L^1_{\partial}$ and that
\begin{equation}\label{e:res}
\mathcal{R}_\lambda=\lim_{N\to\infty}R(\lambda,\mathcal{A})\sum_{n=0}^N (\mathcal{B}R(\lambda,\mathcal{A}))^n=\sum_{n=0}^{\infty} R(\lambda,\mathcal{A})(\mathcal{B}R(\lambda,\mathcal{A}))^n.
\end{equation}
We also have
\[
\lim_{r\uparrow 1}R(\lambda,{\mathcal{G}_r}_|)(f,0)=R(\lambda,\mathcal{G})(f,0),\quad  f\in L^1.
\]
Thus $R(\lambda,\mathcal{G})$ is given by the part ${\mathcal{R}_{\lambda}}_{|}$ of the operator $\mathcal{R}_\lambda$ in $L^1\times\{0\}$, where $\mathcal{R}_\lambda$ is defined by \eqref{e:res}.
Since
\[
R(\lambda,\mathcal{A})\sum_{n=0}^N (\mathcal{B}R(\lambda,\mathcal{A}))^n(\lambda \mathcal{I}-\mathcal{A})(f,0)=\mathcal{I}(f,0)+R(\lambda,\mathcal{A})\sum_{n=0}^{N-1} (\mathcal{B}R(\lambda,\mathcal{A}))^n\mathcal{B}(f,0)
\]
for all $N$, we see that
\[
\mathcal{R}_\lambda (\lambda \mathcal{I}-\mathcal{A}-\mathcal{B})(f,0)=(f,0)
\]
for $f\in \mathcal{D}$, by \eqref{e:res}. Now if $f\in \mathcal{D}(A_{\Psi})$ then $(\lambda \mathcal{I}-\mathcal{A}-\mathcal{B})(f,0)\in L^1\times \{0\}$, implying that $\mathcal{G}$ is an extension of the operator $(\mathcal{A}+\mathcal{B})_{|}$. Thus $(G,\mathcal{D}(G))$ is an extension of the operator $(A_{\Psi}+B,\mathcal{D}(A_{\Psi}))$. Finally, using the formula for $\mathcal{R}_\lambda$ and noting that
\[
R(\lambda,\mathcal{A})\mathcal{B}(f,0)=(R(\lambda,A_0)Bf+\Psi(\lambda)\Psi f,0),\quad f\in D,
\]
we conclude that \eqref{d:resg} holds true.
\end{proof}

\subsection{Characterization of the generator of the perturbed semigroup}\label{s:charg}

We use the notation from Section~\ref{s:ibp}. The operators
$\mathcal{A}$ and $\mathcal{B}$ are as in \eqref{d:AB} and  $(A_{\Psi},\mathcal{D}(A_\Psi))$ is defined by \eqref{d:dombou}.
We begin by noting that Theorem~\ref{t:gmain} together with Remark~\ref{r:non} and Lemma~\ref{t:geth} implies the following characterization.
\begin{corollary}\label{c:gen}
Assume conditions \eqref{S1}--\eqref{S4}.
If the operator  $(\mathcal{A}+\mathcal{B},\mathcal{D}(\mathcal{A}))$  is resolvent positive on $L^1\times L^1_{\partial}$
then  $(A_{\Psi}+B,\mathcal{D}(A_{\Psi}))$ is the generator of a substochastic semigroup on $L^1$.  \end{corollary}

We need the following lemma giving conditions for the operator $(A_{\Psi},\mathcal{D}(A_\Psi))$ 
to be resolvent positive.

\begin{lemma}\label{l:APsi} Assume conditions \eqref{S1}--\eqref{S4}. Let  $\lambda>0$.
Then  $\lambda\in\rho(A_{\Psi})$ if and only if the operator $ I_{\partial }-\Psi \Psi(\lambda)$ is invertible, where $I_{\partial }$ is the identity operator on $L^1_{\partial}$. In that case, the resolvent operator of $(A_{\Psi},\mathcal{D}(A_{\Psi}))$ at $\lambda$ is given by
\begin{equation}\label{eq:resbou}
 R(\lambda,A_{\Psi})f=(I+\Psi(\lambda)(I_{\partial }-\Psi \Psi(\lambda))^{-1}\Psi)R(\lambda,A_0)f,\quad f\in L^1.
\end{equation}
Moreover,
\[
\|\lambda R(\lambda,A_{\Psi})\|\le 1\quad \text{and}\quad \|BR(\lambda,A_{\Psi})\|\le 1.
\]
\end{lemma}
\begin{remark} Since the operator $\Psi\Psi(\lambda)$ is a positive contraction for $\lambda>0$, the operator $I_{\partial}-\Psi\Psi(\lambda)$ is invertible with positive inverse if and only if
\begin{equation}\label{c:rpsilam0}
\lim_{n\to \infty}\|(\Psi\Psi(\lambda))^n\|=0.
\end{equation}
This together with Remark \ref{r:non}\eqref{r:nona} implies that the operator $A_{\Psi}$ is resolvent positive.
\end{remark}
\begin{proof} For the proof of the first part see \cite[Section 3.3.4]{rudnickityran17}.
Since   $R(\lambda,A_{\Psi})f\in \mathcal{D}(A_{\Psi})$ for $f\in L^1$, we have
\[
AR(\lambda,A_{\Psi})f=A_{\Psi}R(\lambda,A_{\Psi})f=\lambda R(\lambda,A_{\Psi})f-f
\]
and $\Psi(R(\lambda,A_{\Psi})f)=\Psi_0(R(\lambda,A_{\Psi})f)$.
Hence, if $f$ is nonnegative, then $R(\lambda,A_{\Psi})f$ is a nonnegative element of $\mathcal{D}$. It follows from \eqref{eq:zero} that
\[
\int_E \big(AR(\lambda,A_{\Psi})f+BR(\lambda,A_{\Psi})f\big)dm\le 0.
\]
Thus, we get
\[
\int_E \lambda R(\lambda,A_{\Psi})f dm -\int_E fdm + \int_E BR(\lambda,A_{\Psi})fdm\le 0.
\]
This shows that both operators $\lambda R(\lambda,A_{\Psi})$ and $BR(\lambda,A_{\Psi})$, being positive operators, have norm smaller or equal to 1.
\end{proof}

It is easily seen that  the following holds.
\begin{lemma}\label{l:rlapb}
Assume conditions \eqref{S1}--\eqref{S4}.
Suppose that $\lambda>0$ is such that the operators $I_{\partial}-\Psi \Psi(\lambda)$ and $I-BR(\lambda,A_{\Psi})$ are invertible.
Then $\mathcal{I}-\mathcal{B}R(\lambda,\mathcal{A})$ is invertible and
\begin{multline}\label{eq:RAB}
R(\lambda,\mathcal{A}+\mathcal{B})(f,f_{\partial})=(R(\lambda,A_{\Psi})(I-BR(\lambda,A_{\Psi}))^{-1}f\\ + (I+R(\lambda,A_{\Psi})(I-BR(\lambda,A_{\Psi}))^{-1}B)\Psi(\lambda)(I_{\partial}-\Psi\Psi(\lambda))^{-1}f_{\partial}, 0).
\end{multline}
\end{lemma}

We now give
one more criterion for $A_{\Psi}+B$ to be the generator. It is a consequence of Lemma~\ref{l:rlapb}, Remark~\ref{r:non} and Corollary~\ref{c:gen}.

\begin{corollary}\label{c:both}
Assume conditions \eqref{S1}--\eqref{S4}. Suppose that there is $\lambda>0$ such that \eqref{c:rpsilam0} holds
and \begin{equation}\label{eq:sspert}
\lim_{n\to \infty}\|(B R(\lambda,A_{\Psi}))^n\|=0.
\end{equation}
Then  $(A_{\Psi}+B,\mathcal{D}(A_{\Psi}))$ is the generator of a substochastic semigroup.
\end{corollary}

Our next goal is to obtain sufficient conditions for the substochastic semigroup  from Theorem~\ref{t:gmain} to be stochastic.

\begin{corollary}\label{c:stochastic}
Assume conditions~\eqref{S1}--\eqref{S2} hold true. If
\begin{equation}\label{eq:zeroo}
\int_{E}(Af+Bf)\,dm+\int_{E_{\partial}}(\Psi f-\Psi_0f)\, dm_{\partial}= 0,\quad f\in  \mathcal{D},f\ge 0,
\end{equation}
then the substochastic semigroup $\{P(t)\}_{t\ge 0}$ from Theorem~\ref{t:gmain}  is stochastic if and only if there is $\lambda>0$ such that
\begin{equation}\label{eq:stoch}
\lim_{n\to\infty}\|(\mathcal{B}R(\lambda,\mathcal{A}))^n(f,0)\|=0,\quad f\in L^1, f\ge0.
\end{equation}
In particular, if conditions \eqref{c:rpsilam0} and \eqref{eq:sspert}  hold for some $\lambda>0$, then $\{P(t)\}_{t\ge 0}$  is a stochastic semigroup and  $(A_{\Psi}+B,\mathcal{D}(A_{\Psi}))$ is its generator.
\end{corollary}
\begin{remark} Note that condition \eqref{eq:zeroo} is necessary for $(A_{\Psi}+B,\mathcal{D}(A_{\Psi}))$ to be the generator of a stochastic semigroup. In the setting of Section~\ref{s:gen} condition \eqref{eq:stoch} is equivalent to
\[
\lim_{n\to\infty}\int_E \mathbb{E}_x(e^{-\lambda \tau_n})f(x)m(dx)=0,\quad f\in L^1, f\ge0,
\]
by Lemma~\ref{l:ltr}. Thus, in particular \eqref{c:finnt} implies \eqref{eq:stoch}.

In applications,  to check condition~\eqref{eq:sspert}
we show that some power of the operator $BR(\lambda,A_{\Psi})$ has the norm strictly smaller than~1, see Section~\ref{s:cell}. Similarly, one can check condition \eqref{c:rpsilam0}.
\end{remark}

\begin{proof}
Recall that a substochastic semigroup with generator $G$ is stochastic if and only if there is $\omega\in \mathbb{R}0$ such that
the operator $\lambda R(\lambda,G)$ is stochastic for  all $\lambda>\omega$. Since $\mathcal{B}R(\lambda,\mathcal{A})$ is a contraction, condition \eqref{eq:stoch} holds for all sufficiently large $\lambda$. Thus $G$ is the generator of a  stochastic semigroup if and only if the operator $\lambda R(\lambda,G)$ is stochastic for all $\lambda$ satisfying \eqref{eq:stoch}.
Observe that
combining \eqref{eq:sum1} with \eqref{eq:zeroo} leads to
\begin{equation}\label{c:stochres}
\|\lambda R(\lambda,\mathcal{A})(g,g_\partial)\|=\|(g,g_\partial)\|-\|\mathcal{B}R(\lambda,\mathcal{A})(g,g_\partial)\|
\end{equation}
for all nonnegative $(g,g_\partial)\in L^1\times L^1_{\partial}$. Hence, for nonnegative $f\in L^1$ and for
\[
(g,g_\partial)=\sum_{n=0}^N (\mathcal{B}R(\lambda,\mathcal{A}))^n(f,0)
\]
we obtain
\[
\lambda\int_E f_Ndm =\int_E fdm-\|(\mathcal{B}R(\lambda,\mathcal{A}))^{N+1}(f,0)\|,
\]
where
\[
f_N=\sum_{n=0}^N (R(\lambda,A_0)B+\Psi(\lambda)\Psi)^nR(\lambda,A_0)f, \quad N\ge 0.
\]
By taking the limit as $N\to \infty$, we see that
\[
\lambda\int_E R(\lambda,G)f dm =\int_E fdm-\lim_{N\to\infty} \|(\mathcal{B}R(\lambda,\mathcal{A}))^{N}(f,0)\|,
\]
since $f_N\uparrow R(\lambda,G)f$ and $\mathcal{B}R(\lambda,\mathcal{A})$ is a contraction.  This completes the proof.
\end{proof}

\subsection{Invariant densities for perturbed semigroups}\label{s:dens}

In this section we
define a linear operator $K $ on the space $L^1\times L^1_{\partial}$ that will correspond to  \eqref{d:operatorK} in the setting of Section~\ref{s:main}.
We also give relationships  between invariant densities of the operator $K$  and invariant densities of the substochastic semigroup  $\{P(t)\}_{t\ge 0}$ from Theorem~\ref{t:gmain};  see \cite[Section 3]{biedrzyckatyran} for the case $\Psi_0=\Psi=0$.
Our next result extends \cite[Theorem 3.6]{tyran09} to the situation studied in this paper.
\begin{theorem}\label{t:opK}
Assume conditions \eqref{S1}--\eqref{S4}.
Define the operator $K \colon L^1\times L^1_{\partial}\to L^1\times L^1_{\partial}$ by
\begin{equation}\label{eq:K}
K (f,f_{\partial})=\lim_{\lambda\downarrow 0}\mathcal{B}R(\lambda,\mathcal{A})(f,f_{\partial}).
\end{equation}
Then $K $ is a substochastic on $L^1\times L^1_{\partial}$. If, additionally, condition \eqref{eq:zeroo} holds then
 $K $ is stochastic if and only if the semigroup $\{S(t)\}_{t\ge 0}$ generated by $(A_0,\mathcal{D}(A_0))$ is {strongly stable}, i.e.
\[
\lim_{t\to\infty} S(t)f=0, \quad f\in L^1.
\]
\end{theorem}
\begin{proof}
The proof of the first part  is as in  \cite{tyran09}.
From \eqref{c:stochres} it follows that
\[
\|K  (f,f_{\partial})\|=\|(f,f_{\partial})\|-\lim_{\lambda\downarrow 0}\lambda\|R(\lambda,\mathcal{A})(f,f_{\partial})\|
\]
for nonnegative $f$ and $f_{\partial}$.
To complete the proof, we use the fact that the mean ergodic
theorem for semigroups \cite[Chapter VIII.4]{yosida78} and additivity of the norm imply that
$\{S(t)\}_{t\ge 0}$ is strongly stable on $L^1$ if and only if
\[
\lim_{\lambda\downarrow 0}\lambda R(\lambda,A_0)f=0,\quad f\in L^1.
\]
Observe that  \eqref{eq:Psilam} implies that $\lim_{\lambda\downarrow 0}\lambda \Psi(\lambda) f_{\partial}=0$ for $f_{\partial}\in L^1_{\partial}$, if $\{S(t)\}_{t\ge 0}$ is strongly stable, completing the proof.
\end{proof}

We have the following extension of \cite[Theorem 3.3]{biedrzyckatyran}.

\begin{theorem}\label{thm:exunifp}
Suppose that conditions \eqref{S1}--\eqref{S4} hold true. Let
$(f,f_{\partial})\in L^1\times L^1_{\partial}$ be an  invariant density for the operator $K $ and let
\begin{equation}\label{eq:esd}
\overline{f}=\sup_{\lambda> 0} (R(\lambda,A_0)f+\Psi(\lambda)f_{\partial}).
\end{equation}
If $\overline{f}\in L^1$ then $\overline{f}/\|\overline{f}\|$ is an invariant density for the semigroup $\{P(t)\}_{t\ge 0}$.
\end{theorem}
\begin{proof} Theorem~\ref{t:gmain} implies that the generator $(G,\mathcal{D}(G))$ of the semigroup $\{P(t)\}_{t\ge 0}$ in an extension of the operator $(A_{\Psi}+B,\mathcal{D}(A_{\Psi}))$. We first show that  $\overline{f}$ as in \eqref{eq:esd} satisfies $\overline{f}\in \mathcal{D}(A_{\Psi})$ and $G\overline{f}\ge 0$.
Let
\[
f_{\lambda}=R(\lambda,A_0)f+\Psi(\lambda)f_{\partial},\quad \lambda>0.
\]
We have $f_\lambda\ge 0$,  $f_\lambda\uparrow \overline{f}$, and $\overline{f}$ is nontrivial.
Since the operator $(\mathcal{A},\mathcal{D}\times \{0\})$ is closed and $\mathcal{A}(f_\lambda,0)=(Af_\lambda,-\Psi_0f_\lambda)=(\lambda f_\lambda -f,-f_{\partial})$, we see that  $\overline{f}\in\mathcal{D}$, $A\overline{f}=-f$ and $\Psi_0\overline{f}=f_{\partial}$. From formula \eqref{eq:K} it follows that $Bf_\lambda\uparrow f$ and $\Psi f_\lambda\uparrow f_{\partial}$ implying that $B\overline{f}\ge f$ and $\Psi\overline{f}\ge f_{\partial}$. Therefore, $ A\overline{f}+B\overline{f}\ge 0$ and $\Psi\overline{f}- \Psi_0\overline{f}\ge 0$. This together with \eqref{eq:zero} gives
\[
\int_{E_{\partial}}(\Psi\overline{f}- \Psi_0\overline{f})dm_{\partial}= 0.
\]
Hence, $\Psi\overline{f}=\Psi_0\overline{f}$ and $G\overline{f}=A\overline{f}+B\overline{f}\ge 0$.
Next, we see that
\[P(t)\overline{f}-\overline{f}=\int_0^t P(s)G\overline{f}ds\ge 0
\]
implying that $P(t)\overline{f}\ge \overline{f}$ for all $t>0$. Since the operator $P(t)$ is a contraction, the result follows. \end{proof}

We also have the following converse of Theorem \ref{thm:exunifp} extending \cite[Corollary 3.11]{biedrzyckatyran}.

\begin{theorem}\label{thm:exunifpi} Assume conditions \eqref{S1}--\eqref{S4}.
Suppose that the semigroup $\{P(t)\}_{t\ge 0}$ has an invariant density ${f}_*$ and that the operator $K $ is stochastic.
If $(B {f}_*,\Psi {f}_*)\in L^1\times L^1_{\partial}$ then $ {f}_*\in \mathcal{D}$, $\|(B {f}_*,\Psi {f}_*)\|>0$,  and  $(B {f}_*,\Psi{f}_*)/\|(B {f}_*,\Psi {f}_*)\|$ is an invariant density for the operator $K $.
\end{theorem}
\begin{proof}
Let $f_0=\lambda  {f}_*$, where $\lambda>0$ is fixed.  We define
\[
(f_N,f_{\partial,N})=\sum_{n=0}^N (\mathcal{B}R(\lambda,\mathcal{A})))^n (f_0,0)\quad \text{and}\quad (\tilde{f}_N,0)=R(\lambda,\mathcal{A})(f_N,f_{\partial,N})\quad N\ge 0.
\]
We have $\tilde{f}_N\uparrow \lambda R(\lambda,G){f}_*= {f}_*$ and   $\mathcal{B}(\tilde{f}_N,0)\le \mathcal{B}(\tilde{f}_{N+1},0)\le  \mathcal{B}( {f}_*,0)$ for all $N$. Since $\mathcal{B}( {f}_*,0)\in L^1\times L^1_{\partial}$, we see that there exists nonnegative  $(f,f_{\partial})\in L^1\times L^1_{\partial}$ such that $\mathcal{B}(\tilde{f}_N,0)\to (f,f_{\partial})$ as $N\to \infty$.
We have
\[
\mathcal{B}(\tilde{f}_N,0)=\mathcal{B}R(\lambda,\mathcal{A})(f_N,f_{\partial,N})=(f_{N+1}-f_0,f_{\partial,N+1})
\]
for all $N$. Thus, $(f_N,f_{\partial,N})\to (f+f_0,f_{\partial})$,
\begin{equation*}
(f_*,0)=R(\lambda,\mathcal{A}) (f,f_{\partial})+R(\lambda,\mathcal{A})(f_0,0),
\end{equation*}
and $f_*\in \mathcal{D}$.
Next, we show that $\|(f,f_{\partial})\|>0$. Suppose, contrary to our claim that $\|(f,f_{\partial})\|=0$.
Then $ {f}_*=\lambda R(\lambda,A_0) {f}_*$, implying that $ {f}_*$ is an invariant density for the semigroup $\{S(t)\}_{t\ge 0}$ generated by the operator $(A_0,\mathcal{D}(A_0))$. By Theorem~\ref{t:opK}, $\{S(t)\}_{t\ge 0}$ is strongly stable, giving $ {f}_*=0$ and leading to a contradiction.
Finally, since $(f,f_{\partial})\le \mathcal{B}({f}_*,0)=(Bf_*,\Psi f_*)$, we see that $\|\mathcal{B}({f}_*,0)\|>0$ and
$
\mathcal{B}({f}_*,0)\le K \mathcal{B}({f}_*,0)+\lambda K ({f}_*,0),
$
where  letting $\lambda\downarrow 0$ completes the proof.
\end{proof}

\section{Proofs of  main results}\label{s:proofs}

We consider the minimal process $X=\{X(t)\}_{t\ge 0}$ with characteristics $(\phi,q,\mathcal{P})$ as described in Section~\ref{s:PDMP} and  such that Assumptions~\ref{a:nons}--\ref{a:jump} from Sections~\ref{s:exi} hold true.
To use results from Section~\ref{s:pert} we take  $L^1=L^1(E,m)$,  $L^1_{\partial}=L^1(\nbound,m^{-})$, $\Psi_0=\gamma^{-}$, and
the operators $A$, $B$, $\Psi$ as described in Theorem~\ref{th:ptsmal} in Section ~\ref{s:gen}.
We check that Theorem~\ref{t:gmain} applies and  provides the existence of a substochastic semigroup that will be the semigroup induced by the minimal process $X$ implying Theorems~\ref{t:mPg} and~\ref{th:ptsmal}.
In Section~\ref{s:inv} we use the results from Section~\ref{s:dens} to prove Theorems~\ref{t:Kkernel}--\ref{c:invconv} from Section~\ref{s:exiinv}.

\subsection{Existence of a substochastic semigroup}\label{s:deter}

In this section we  check that  assumptions of Theorem~\ref{th:ptsmal} imply
conditions \eqref{S1}--\eqref{S4} of Theorem~\ref{t:gmain} leading to the following result.

\begin{theorem}\label{th:ptsmall} Suppose that Assumptions~\ref{a:nons}--\ref{a:jump} hold. Let $B$ and $\Psi$ be as in \eqref{d:part} and $(A,\mathcal{D})$ be defined by \eqref{d:operatorA}--\eqref{d:domainopA}.  Then 
there exists a substochastic semigroup $\{P(t)\}_{t\ge 0}$ with generator $(G,\mathcal{D}(G))$ being an extension of the operator $(A_{\Psi}+B,\mathcal{D}(A_{\Psi}))$ where $\mathcal{D}(A_{\Psi})=\{f\in \mathcal{D}:\gamma^{-}f=\Psi f\}$. The resolvent operator of $G$ at $\lambda>0$ is given by
\eqref{d:resg}. Moreover, condition \eqref{eq:zeroo} holds.
\end{theorem}

Before we give the proof of Theorem~\ref{th:ptsmall}, we first provide a general formula for the right inverse $\Psi(\lambda)$ introduced in condition~\eqref{S2} in Section~\ref{s:ibp}.
 Suppose that Assumptions~\ref{a:nons}--\ref{a:dive} hold and, for each $\lambda>0$,  define
\begin{equation}\label{e:psilam}
\Psi(\lambda)f_{\partial}(x)=e^{-\lambda \nlife(x)-\int_0^{\nlife(x)}\vartheta(\phi_{-r}(x))dr}f_{\partial}(\phi_{-\nlife(x)}(x))J_{-\nlife(x)}(x),\quad x\in E, f_{\partial}\in L^{1}(\nbound ,m^{-}),
\end{equation}
 where the right-hand side of \eqref{e:psilam} is equal to zero if $\nlife(x)=\infty$.

\begin{lemma}\label{r:psil}
Let $f_\partial \in L^1(\nbound ,m^{-})$ and $\lambda>0$. If  $\Psi(\lambda)f_{\partial}$ is as in~\eqref{e:psilam} then  $\gamma^{-}\Psi(\lambda)f_{\partial}=f_{\partial}$  and
\begin{equation}\label{e:gppsilam}
\gamma^{+}\Psi(\lambda)f_{\partial}(z)=e^{-\lambda \nlife(z)-\int_0^{\nlife(z)} \vartheta(\phi_{-r}(z))dr}f_{\partial}(\phi_{-\nlife(z)}(z))J_{-\nlife(z)}(z),\quad z\in \pbound .
\end{equation}
Moreover,  $\gamma^{+}\Psi(\lambda)f_{\partial}\in L^1(\pbound,m^+)$,
\begin{equation}\label{e:Grb0}
\int_{\pbound }\gamma^{+}\Psi(\lambda)f_\partial(z)m^{+}(dz)=\int_{\nbound }e^{-\int_0^{\life(z)}(\lambda+\vartheta(\phi_{r}(z)))dr}f_{\partial}(z) m^{-}(dz)
\end{equation}
and
\begin{equation}\label{e:Grb}
\int_{\pbound }\gamma^{+}\Psi(\lambda)f_{\partial}(z)m^{+}(dz)+\int_{E}(\lambda+\vartheta(x))\Psi(\lambda)f_{\partial}(x)m(dx)=\int_{\nbound }f_{\partial}(z)m^{-}(dz).
\end{equation}
\end{lemma}
\begin{proof} Let $f=\Psi(\lambda)f_{\partial}$.
Since $\nlife(\phi_s(z))=s$ for $s<\life(z)$ and $z\in \nbound $, we get, by \eqref{e:psilam} and \eqref{d:jacob},
\begin{equation}\label{e:psilgm}
f(\phi_s(z))J_s(z)=e^{-\lambda s-\int_0^{s}\vartheta(\phi_{r}(z))dr}f_{\partial}(z).
\end{equation}
Thus $\gamma^{-}f(z)=f_{\partial}(z)$ for $z\in \nbound $, by letting $s\to 0$ in  \eqref{e:psilgm}.
Similarly, since $\nlife(\phi_{-s}(z))=\nlife(z)-s$ for $z\in \pbound $ and $s<\nlife(z)<\infty$, we obtain
\[
f(\phi_{-s}(z))J_{-s}(z)=e^{-\int_s^{\nlife(z)}(\lambda + \vartheta(\phi_{-r} (z)))dr}f_{\partial}(\phi_{-\nlife(z)}(z))J_{-\nlife(z)}(z),
\]
showing that \eqref{e:gppsilam} holds.

Assume now that $f_{\partial}\ge 0$. Then $f\ge 0$. Recall from \eqref{d:spaceEpm} that $E\setminus E_{-}\subset \nbound \cup\{x: \nlife(x)=\infty\}$. We have $m(\nbound )=0$ and $f(x)=0$ if $\nlife(x)=\infty$. Thus, by \eqref{e:eminus} and \eqref{e:psilgm}, we obtain
\[
\begin{split}
\int_{E}(\lambda+\vartheta(x))f(x)m(dx)&=\int_{\nbound }\int_0^{\life(z)}(\lambda+\vartheta(\phi_s (z)))e^{-\lambda s-\int_0^{s}\vartheta(\phi_{r}(z))dr}f_{\partial}(z)
ds\,m^{-}(dz).
\end{split}
\]
It follows from Assumption \ref{a:varp} that
\[
 (\lambda+\vartheta(\phi_s(x)))e^{-\int_0^s(\lambda + \vartheta(\phi_r(x)))dr}=-\frac{d}{ds}e^{-\int_0^s(\lambda + \vartheta(\phi_r(x)))dr}
\]
for Lebesgue almost every $s$ and for all $x$. Hence, for all $x$ we have
\begin{equation}\label{e:phiint}
\int_0^{\life(x)} (\lambda+\vartheta(\phi_s(x)))e^{-\int_0^s(\lambda + \vartheta(\phi_r(x)))dr}ds= 1-e^{-\int_0^{\life(x)}(\lambda + \vartheta(\phi_r(x)))dr}.
\end{equation}
Therefore
\[
\begin{split}
\int_{E}(\lambda+\vartheta(x))f(x)m(dx)&=\int_{\nbound } \left(1-e^{-\lambda \life(z)-\int_0^{\life(z)}\vartheta(\phi_{r}(z))dr}\right)f_{\partial}(z) m^{-}(dz)\le \|f_{\partial}\|.
\end{split}
\]
Observe that for any $\lambda >0$ and nonnegative measurable $g$  we have (\cite[Proposition 3.12]{arlottibanasiaklods07})
\begin{equation}\label{e:gammam}
\int_{\pbound } e^{-\lambda \nlife(z)} g(\phi_{-\nlife(z)}(z))J_{-\nlife(z)}(z)\,m^{+}(dz)=\int_{\nbound }e^{-\lambda \life(z)} g(z)m^{-}(dz).
\end{equation}
By applying \eqref{e:gammam} to $g(z)=e^{-\int_0^{\life(z)}\vartheta(\phi_{r}(z))dr}f_{\partial}(z)$ we see that \eqref{e:Grb0} holds, implying~\eqref{e:Grb}.
Now decomposing $f_\partial$ as the difference of a positive and a negative part,  completes the proof.
\end{proof}

Our next result shows that condition \eqref{S2} from Section~\ref{s:ibp} holds.

\begin{lemma}\label{t:psil}
Let  $(A,\mathcal{D})$ be given by \eqref{d:operatorA}--\eqref{d:domainopA} and let $\Psi_0(f)=\gamma^{-}f$ for $f\in \mathcal{D}$. Then for any $\lambda>0$,  $\Psi(\lambda)$ given by \eqref{e:psilam} is the right-inverse of the operator $\Psi_0$  restricted to the nullspace of $\lambda -A$.
\end{lemma}

\begin{proof}
Let $f_{\partial}\in L^1(\nbound ,m^{-})$ and $f=\Psi(\lambda)f_{\partial}$ with $\lambda>0$.
Lemma~\ref{r:psil} implies that $f\in L^1(E,m)$, $\vartheta f\in L^1(E,m)$ and $f\in \mathcal{D}(\gamma^{\pm})$.
It remains  to show that $f\in \Ker(\lambda -A)$, or, equivalently, that $f\in \mathcal{D}_{\max}$ and $\Tm f=(\lambda+\vartheta) f$.   To this end, it is enough to prove that for any test function $\psi  \in\mathfrak{N}$ we have
\[
\int_{E_{-}} (\lambda+\vartheta(x))f(x)\psi  (x)m(dx)=\int_{E_{-}}f(x)\frac{d(\psi  \circ \phi_t)}{dt}\Big|_{t=0}(x)m(dx),
\]
where we use the fact  that $f(x)=0$ for $m$-a.e. $x\in E\setminus E_{-}$.
By the change of variables \eqref{e:eminus}
\[
\begin{split}
\int_{E_{-}} (\lambda+\vartheta(x))f(x)\psi  (x)m(dx)&=\int_{\nbound }\int_0^{\life(z)}-\frac{d}{ds}\left(e^{-\int_0^{s}(\lambda +\vartheta(\phi_{r}(z)))dr} \right)\psi  (\phi_s(z))ds f_{\partial}(z) m^{-}(dz).
\end{split}
\]
Integration by parts leads to
\[
\int_0^{\life(z)}-\frac{d}{ds}\left(e^{-\int_0^{s}(\lambda+\vartheta(\phi_{r}(z)))dr} \right)\psi  (\phi_s(z))\,ds=\int_0^{\life(z)}e^{-\int_0^{s}(\lambda+\vartheta(\phi_{r}(z)))dr} \frac{d}{ds}\left(\psi  (\phi_s(z))\right)ds,
\]
since $e^{-\lambda \life(z)}\psi  (\phi_{s}(z))\to 0$ as $s\to \life(z)$ and $\psi  (z)=0$ for $z\in \nbound $.  This together with \eqref{e:psilgm} gives
\[
\begin{split}
\int_{E_{-}} (\lambda+\vartheta(x))f(x)\psi  (x)m(dx)
&=\int_{\nbound }\int_0^{\life(z)} J_s(z)f(\phi_s(z))\frac{d(\psi  \circ \phi_t)}{dt}\Big|_{t=0}(\phi_s(z))\,ds.
\end{split}
\]
Using again the change of variables \eqref{e:eminus}, completes the proof.
\end{proof}

\begin{proof}[Proof of Theorem~\ref{th:ptsmall}]
It follows from Theorem~\ref{l:Green0} that for $f\in \mathcal{D}$  we have $\gamma^+f\in L^1(\pbound ,m^+)$ and
\begin{equation}\label{e:Afin}
\int_{E} Af(x)m(dx)=\int_{\nbound }\gamma^{-}f(x)m^{-}(dx)-\int_{\pbound }\gamma^{+}f(x)m^{+}(dx)-\int_{E}\vartheta(x)f(x)m(dx).
\end{equation}
If $f$ is nonnegative then it follows from \eqref{e:jump} that
\[
\begin{split}
\int_{E}Bf(x)m(dx)+\int_{\nbound }\Psi f(x)m^{-}(dx)&=\int_{E} \jump (x,E)\vartheta(x)f(x)m(dx) +\int_{\pbound }\jump (x,E) \gamma^{+}f(x)m^{+}(dx)\\
&\le \int_{E}\vartheta(x)f(x)m(dx)+\int_{\pbound }\gamma^{+}f(x)m^{+}(dx),
\end{split}
\]
where equality holds if
$\jump(x,E)=1$ for all $x\in E\cup\pbound $.
This together with Theorem~\ref{t:gs0t} and Lemma~\ref{t:psil} shows that conditions \eqref{S1}--\eqref{S4} in Section~\ref{s:ibp} are satisfied. Theorem~\ref{t:gmain} now completes the proof.
\end{proof}

We conclude this section with the following result that will be needed in the next sections.
\begin{lemma}\label{r:gamA0} Suppose that Assumptions~\ref{a:nons}--\ref{a:dive} hold.  Let  $(A_0,\mathcal{D}(A_0))$ be the generator of the substochastic semigroup in \eqref{d:ss0}.
For any nonnegative $f\in L^1(E,m)$ and $\lambda>0$ we have
\[
R(\lambda,A_0)f(x)=\int_0^{\nlife(x)}e^{-\lambda t-\int_0^{t}\vartheta(\phi_{-r}(x))dr}f(\phi_{-t}(x))J_{-t}(x)dt, \quad x\in E,
\]
 and
\[
\gamma^{+}R(\lambda,A_0)f(z)=\int_{0}^{\nlife(z)}e^{-\lambda t-\int_0^{t}\vartheta(\phi_{-r}(z))dr}f(\phi_{-t}(z))J_{-t}(z)dt,\quad z\in \pbound .
\]
Moreover,
\begin{equation*}
\int_{\pbound }\gamma^{+}R(\lambda,A_0)f(z)m^{+}(dz)=\int_{E} e^{-\lambda \life(x)-\int_{0}^{\life(x)}\vartheta(\phi_{r}(x))dr}f(x)m(dx)
\end{equation*}
and
\begin{equation*}
\int_{\pbound }\gamma^{+}R(\lambda,A_0)f(z)m^{+}(dz)+ \int_{E}(\lambda +\vartheta(x))R(\lambda,A_0)f(x)m(dx)=\int_{E}f(x)m(dx).
\end{equation*}
\end{lemma}
\begin{proof} Since
\[
R(\lambda,A_0)f(x)=\int_0^\infty e^{-\lambda t}S(t)f(x)dt,
\]
the first formula follows from \eqref{d:ss0}. This together with \eqref{d:jacob} and the monotone convergence theorem  implies that the second formula is valid.
Fubini's theorem together with conditions \eqref{eq:T0St} and \eqref{e:phiint} gives
\[
\begin{split}
\int_E (\lambda +\vartheta(x))R(\lambda,A_0)f(x)m(dx)&=\int_0^\infty e^{-\lambda t}\int_E 1_{[0,\life(x))}(t) e^{-\int_0^t \vartheta(\phi_r(x))dr}(\lambda+\vartheta(\phi_t(x)))f(x)m(dx)dt\\
&=\int_E \left(1-e^{-\int_0^{\life(x)}(\lambda + \vartheta(\phi_r(x)))dr}\right)f(x)m(dx).
\end{split}
\]
It follows from \eqref{e:eplus} that
\[
\int_{\pbound }\gamma^{+}R(\lambda,A_0)f(z)m^{+}(dz)=\int_{E_{+}} e^{-\lambda \life(x)-\int_{0}^{\life(x)}\vartheta(\phi_{r}(x))dr}f(x)m(dx).
\]
Finally, we have $e^{-\lambda \life(x)}=0$ for $x\in E_{+\infty}=\{x\in E:\life(x)=+\infty\}$ and $E=E_{+}\cup  E_{+\infty}$, which
 completes the proof.
\end{proof}

\subsection{Proofs of Theorems \ref{t:mPg} and \ref{th:ptsmal}}\label{s:proofm1}
In the proof of  Theorem~\ref{t:mPg}  we will show that the semigroup $\{P(t)\}_{t\ge0}$ from Theorem~\ref{th:ptsmall} is the  semigroup induced by the process $X=\{X(t)\}_{t\ge 0}$ with characteristics $(\phi,\vartheta,\mathcal{P})$.
Recall that for any $x\in E$ and $\F \in \mathcal{B}(E)$ the transition function is  $P(t,x,\F )=\mathbb{P}_x(X(t)\in \F ,t<\tau_\infty)$, where $\mathbb{P}_x$ is the distribution of the process $X$ starting at $x$ and $\tau_\infty$ is the explosion time. Thus
\begin{equation*}
P(t,x,\F )=\sum_{n=0}^\infty \mathbb{P}_x(X(t)\in \F ,\tau_n\le t<\tau_{n+1}),\quad x\in E,\F \in \mathcal{B}(E),
\end{equation*}
where $\tau_n$ are the consecutive jump times of the process.
First,
for $\lambda>0$, $x\in E$ and $\psi \in B(E)$ we define
\begin{equation*}
U_\lambda \psi (x)=\int_{0}^\infty e^{-\lambda t}\int_E \psi (y)P(t,x,dy)dt
\end{equation*}
and we rewrite it with the help of the embedded discrete time Markov chain describing consecutive jump times and post-jump positions.
We define the transition kernel as in \cite[Equation (4.3)]{costadufour08}
\begin{equation*}
N(x,\F \times J)=\mathbb{E}_x[\mathbf{1}_{\F}(X(\tau_1))\mathbf{1}_J(\tau_1)],\quad x\in E,
\end{equation*}
for $\F \in \mathcal{B}(E)$, $J\in \mathcal{B}(\mathbb{R}_{+})$. The strong Markov property of the process $\{X(t)\}_{t\ge 0}$ at $\tau_n$ implies that the sequence  $(X(\tau_n),\tau_n)$, $n\ge 0$, is a (sub)Markov chain on $E\times \mathbb{R}_+$ satisfying
the iterative formula
\[
\begin{split}
N^{n}(x,\F \times J)&=\mathbb{P}_x(X(\tau_n)\in \F ,\tau_n\in J)
=\int_{E\times \mathbb{R}_+}N^{n-1}(y,\F \times (J-s))N(x,dy,ds)
\end{split}
\]
for $n\ge 1$, $N^{1}=N$, and $N^{0}(x,\F \times J)=\mathbf{1}_{\F} (x)\delta_0(J)$ for $\F \in \mathcal{B}(E)$, $J\in \mathcal{B}(\mathbb{R}_{+})$.
Let $\psi \in B(E)$. We define
\begin{equation*}
T_0(t)\psi (x)=\mathbb{E}_x [\psi (X(t))\mathbf{1}_{\{t<\tau_{1}\}}]=\psi (\phi_t(x))e^{-\int_{0}^t \vartheta(\phi_r(x))dr}\mathbf{1}_{[0,\life(x))}(t)
\end{equation*}
and its Laplace transform
\begin{equation}\label{d:Ul0}
U_\lambda^0\psi (x)=\int_0^{\infty} e^{-\lambda t} T_0(t)\psi (x)dt=\int_0^{\life(x)}\psi (\phi_t(x))e^{-\int_{0}^t (\lambda +\vartheta(\phi_r(x)))dr}dt,\quad x\in E,\lambda>0.
\end{equation}
For each $n$, by the strong Markov property at $\tau_n$, we obtain
\[
\begin{split}
\mathbb{E}_x[\psi (X(t))\mathbf{1}_{\{\tau_n\le t<\tau_{n+1}\}}]&=\mathbb{E}_x[\psi (\phi_{t-\tau_n}(X(\tau_n)))\mathbf{1}_{\{\tau_n\le t<\tau_{n+1}\}}] =\int_{E\times [0,t]}T_0(t-s)\psi (y)N^{n}(x,dy,ds).
\end{split}
\]
Consequently, for $\lambda>0$, $x\in E$ and $\psi \in B(E)$ we have
\begin{equation}\label{d:Ulam}
U_\lambda \psi (x)
=\sum_{n=0}^{\infty}\int_{0}^\infty\int_{E}e^{-\lambda s}U_\lambda^0\psi (y)N^{n}(x,dy,ds)=\sum_{n=0}^{\infty} \mathcal{K}_{\lambda}^n U_{\lambda}^0\psi (x),
\end{equation}
where
\[
\mathcal{K}_{\lambda}^n\psi (x)=\int_{E}\int_0^{\infty} e^{-\lambda s}\psi (y)N^{n}(x,dy,ds)=\mathbb{E}_x(e^{-\lambda \tau_n}\psi (X(\tau_n))),\quad n\ge 0.
\]
Note that $\mathcal{K}_{\lambda}^n$ is the $n$th iterate of the operator
\begin{equation}\label{e:Ql}
\mathcal{K}_{\lambda}\psi (x)=\int_E \psi (y)\mathcal{K}_{\lambda}(x,dy), \quad x\in E, \psi \in B(E),
\end{equation}
where the transition kernel $\mathcal{K}_\lambda$ is given by
\begin{equation}\label{e:Q2}
\begin{split}
\mathcal{K}_{\lambda}(x,\F )
=\int_0^{\life(x)} e^{-\lambda s-\int_0^s \vartheta(\phi_r(x))dr }\vartheta(\phi_s(x))\jump(\phi_s(x),\F )ds +e^{-\lambda \life(x)-\int_0^{\life(x)} \vartheta(\phi_r(x))dr }\jump(\phi_{ \life(x)}(x),\F )
\end{split}
\end{equation}
for all $x\in E$ and $\F \in \mathcal{B}(E)$. Note that $\mathcal{K}_1$ corresponds to $R$ in \cite[Equation (2.5)]{costadufour08}.

In what follows we use the following duality notation
\[
\langle (f,f_\partial),  \psi  \rangle=\int_{E} f(x)\psi (x)m(dx)+\int_{\nbound }f_{\partial} (x)\psi (x)m^{-}(dx)
\]
for $f\in L^1(E,m)$, $f_{\partial}\in L^1(\nbound ,m^{-})$, and bounded  measurable functions $\psi \colon E\to \mathbb{R}$.
We let $\mathcal{A}$ and $\mathcal{B}$ be defined as in \eqref{d:AB} where the operators $A$, $B$, $\Psi$ are as described in Theorem~\ref{th:ptsmal} and $\Psi_0=\gamma^{-}$.

\begin{lemma}\label{l:ltr} Let $\mathcal{B}R(\lambda,\mathcal{A})$ be  as in \eqref{eq:BRA} and $\mathcal{K}_\lambda$ as in \eqref{e:Ql}. Then for any nonnegative $(f,f_{\partial})\in L^1(E,m)\times L^1(\nbound ,m^{-})$ and any nonnegative measurable $\psi $ we have \[
\langle \mathcal{B}R(\lambda,\mathcal{A})(f,f_\partial), \psi \rangle = \langle (f,f_\partial), \mathcal{K}_\lambda \psi  \rangle,\quad \lambda>0.
\]
\end{lemma}
\begin{proof} Let $\Phi_x$ be as in \eqref{d:Phi}. From \eqref{e:Q2} it follows that
\begin{align}\label{e:adQl}
\int_E f(x)\mathcal{K}_{\lambda}(x,dy)&=\int_0^\infty\int_E  f(x)e^{-\lambda s}\Phi_x(s)\vartheta(\phi_s(x))\jump(\phi_s(x),dy)\,m(dx)\,ds\\
\nonumber &\quad+\int_{E_{+}} f(x)e^{-\lambda \life(x)}\Phi_x(\life(x)^{-})\jump(\phi_{\life(x)}(x),dy)m(dx).
\end{align}
We begin by rewriting the first integral in the right-hand side of \eqref{e:adQl}.
For each $s>0$, using \eqref{eq:T0St}, we get
\[
\begin{split}
\int_E  f(x)\Phi_x(s)\vartheta(\phi_s(x))\jump(\phi_s(x),dy)m(dx)=\int_E  S(s)f(x)\vartheta(x)\jump(x,dy) m(dx).
\end{split}
\]
Hence,
\[
\int_0^\infty\int_E  e^{-\lambda s}f(x)\Phi_x(s)\vartheta(\phi_s(x))\jump(\phi_s(x),dy)\,m(dx)\,ds=\int_E  R(\lambda,A_0)f(x)\vartheta(x)\jump(x,dy) m(dx).
\]
To rewrite the second integral in \eqref{e:adQl}, we make use of \eqref{e:eplus}  to get
\[\begin{split}
\int_E  f(x)e^{-\lambda \life(x)}\Phi_x(\life(x)^{-})\jump(\phi_{\life(x)}(x),dy)m(dx)=\int_{\pbound }f_{\partial^+}(z)\jump(z,dy)m^{+}(dz) ,
\end{split}
\]
where
\[
f_{\partial^+}(z)=\int_{0}^{\nlife(z)}e^{-\lambda s}e^{-\int_0^s\vartheta(\phi_r(\phi_{-s}(z)))dr}f(\phi_{-s}(z))J_{-s}(z) ds.
\]
This together with \eqref{e:jump} leads to
\begin{multline*}
\int_E  R(\lambda,A_0)f(x)\vartheta(x)\int_E \psi (y)\jump(x,dy) m(dx)+\int_{\pbound }f_{\partial^+}(z)\int_E \psi (y)\jump(z,dy)m^{+}(dz)\\
=\big\langle \big(P_0(\vartheta R(\lambda,A_0)f,f_{\partial^+})),P_{\partial}(\vartheta R(\lambda,A_0)f,f_{\partial^+}))\big),\psi \big\rangle.
\end{multline*}
Since $f_{\partial^+}=\gamma^{+}R(\lambda,A_0)f$, by Lemma~\ref{r:gamA0}, we obtain
\begin{equation}\label{e:fgQ}
\int_E f(x)\int_E \psi (y)\mathcal{K}_{\lambda}(x,dy)m(dx)=\langle (B(R(\lambda,A_0)f),\Psi(R(\lambda,A_0)f),\psi \rangle. \end{equation}

Similarly, we have
\[
\begin{split}
\int_{\nbound } f_\partial(x)\mathcal{K}_{\lambda}(x,dy)m^{-}(dx)&=\int_{\nbound }\int_0^{\life(x)}f_{\partial}(x)e^{-\lambda s}e^{-\int_{0}^s\vartheta(\phi_r(x))dr}\vartheta(\phi_s(x))\jump(\phi_s(x),dy)\,ds\,m^{-}(dx)\\
&\quad+\int_{\nbound } f_\partial(x) e^{-\lambda \life(x)}\Phi_x(\life(x)^{-})\jump(\phi_{\life(x)}(x),dy)m^{-}(dx)
\\
&=\int_{E}\Psi(\lambda)f_{\partial}(x)\vartheta(x)\jump(x,dy)m(dx) + \int_{\pbound }\gamma^{+}\Psi(\lambda)f_{\partial}(x)\jump(x,dy)m^{+}(dx),
\end{split}
\]
where we used \eqref{e:eminus} and \eqref{e:gammam}.
Finally, we conclude from \eqref{e:jump}  that
\[
\int_{\nbound } f_\partial(x)\int_E \psi (y)\mathcal{K}_{\lambda}(x,dy)m(dx)=\big\langle (B\Psi(\lambda)f_{\partial},\Psi\Psi(\lambda)f_{\partial}),\psi \big\rangle.
\]
This together with \eqref{e:fgQ} completes the proof.
\end{proof}

Now we are prepared to give the

\begin{proof}[Proof of Theorem~\ref{t:mPg}]
Assume that $f,f_{\partial}, \psi $ are measurable and nonnegative.
Observe that  we have
\begin{equation}\label{e:U0lambda}
\langle R(\lambda,\mathcal{A})(f,f_{\partial}),\psi \rangle=\langle (f,f_{\partial}), U_{\lambda}^0\psi \rangle
\end{equation}
where $U_{\lambda}^0$ is as in \eqref{d:Ul0}. It follows from Lemma~\ref{l:ltr} that 
\[
\langle R(\lambda,\mathcal{A})\mathcal{B}R(\lambda,\mathcal{A})(f,f_{\partial}),\psi \rangle=\langle (f,f_{\partial}), \mathcal{K}_\lambda U_{\lambda}^0\psi \rangle.
\]
Consequently, for any $n\ge 1$ we obtain
\[
\langle R(\lambda,\mathcal{A})(\mathcal{B}R(\lambda,\mathcal{A}))^n(f,f_{\partial}),\psi \rangle=\langle (f,f_{\partial}), \mathcal{K}_\lambda^n U_{\lambda}^0\psi \rangle.
\]
By the Lebesgue monotone convergence theorem,
\[
\lim_{N\to \infty} \langle \sum_{n=0}^N R(\lambda,\mathcal{A})(\mathcal{B}R(\lambda,\mathcal{A}))^n(f,f_{\partial}),\psi \rangle=\langle \mathcal{R}_\lambda(f,f_{\partial}),\psi \rangle
\]
and
\[
\lim_{N\to \infty} \langle (f,f_{\partial}), \sum_{n=0}^N \mathcal{K}_\lambda^n U_{\lambda}^0\psi \rangle =\langle (f,f_{\partial}),U_\lambda \psi \rangle, \]
where $\mathcal{R}_\lambda$ is as in \eqref{e:res} and $U_\lambda \psi $ is as in \eqref{d:Ulam}.
This shows that
\begin{equation*}
\int_E  R(\lambda,G)f(x)\psi (x) m(dx)=\int_{E} f(x) U_\lambda \psi (x) m(dx) ,
\end{equation*}
since $\mathcal{R}_\lambda(f,0)=( R(\lambda,G)f,0)$ for $f\in L^1(E,m)$.
The process $\{X(t)\}_{t\ge 0}$ has right-continuous sample paths by construction.
Let $\psi \in \Lip(E)$, where $\Lip(E)$ is the set of bounded globally Lipschitz functions $\psi \colon E\to \mathbb{R}$.
Thus, we get
\[
\lim_{s\to t^{+}}\mathbb{E}_x(\psi (X(s)))=\mathbb{E}_x(\psi (X(t))),\quad x\in E,t\ge 0, \psi \in \Lip(E),
\]
and we conclude that the function
\[
t\mapsto \int_{E} f(x) \int_E \psi (y)P(t,x,dy) m(dx)
\]
is right-continuous for any $\psi \in \Lip(E)$ and any nonnegative $f\in L^1$. We also have
\[
\int_E  R(\lambda,G)f(x)\psi (x) m(dx)=\int_0^\infty e^{-\lambda t}\int_E P(t)f(x)\psi (x) m(dx)dt
\]
and the function
\[
t\mapsto \int_E P(t)f(x)\psi (x) m(dx)
\]
is continuous. Hence, by the uniqueness of the Laplace tranform, we obtain
\begin{equation}\label{e:eqLT}
\int_{E} P(t)f(x)\psi (x)m(dx)=\int_{E}\int_E \psi (y)P(t,x,dy)f(x)m(dx)
\end{equation}
for all $t>0$, nonnegative $f\in L^1$ and $\psi \in \Lip(E)$. Finally, we can approximate indicator functions of closed sets by functions from $\Lip(E)$.
Thus equality \eqref{e:eqLT} holds for all $\psi $ being indicator functions of closed subsets of $E$. Since two finite Borel measures are uniquely defined through their values on closed sets, we conclude that \eqref{e:eqLT} holds for $\psi =\mathbf{1}_{\F }$, $\F \in \mathcal{B}(E)$. This completes the proof of Theorem~\ref{t:mPg}.
\end{proof}

Finally, we prove  our results from Section~\ref{s:gen}.

\begin{proof}[Proof of Theorem~\ref{th:ptsmal}]
Theorem \ref{th:ptsmall} together with Theorem~\ref{t:mPg}  implies that the generator $(G,\mathcal{D}(G))$ of the induced semigroup $\{P(t)\}_{t\ge0}$ is an extension of the operator $(A_{\Psi}+B,\mathcal{D}(A_\Psi))$. Now, if $\mathcal{D}(G)=\mathcal{D}(A_\Psi)$ then $G=A_{\Psi}+B$ is the generator of a substochastic semigroup satisfying
\[
\int_E Gf dm=0\quad \text{for }f\in \mathcal{D}(G),f\ge 0
\]
by Theorem \ref{th:ptsmall} and \eqref{eq:zeroo}. Hence, the induced semigroup is stochastic.
\end{proof}

\begin{proof}[Proof of Corollary~\ref{c:sto}]
Let $\overline{\vartheta}$ be the upper bound for $\vartheta$ and let $c$ be the lower bound for $\life$ on $\nbound $.  Observe that for nonnegative $(f,f_{\partial})\in L^1(E,m)\times L^1(\nbound ,m^{-})$ and $\lambda>0$ we have
\[
\|\mathcal{B}R(\lambda,\mathcal{A})(f,f_{\partial})\|\le \left(\frac{\overline{\vartheta}}{\lambda}+e^{-\lambda c}\right)\|(f,f_\partial)\|.
\]
This shows that \eqref{eq:ssper} holds and Corollaries \ref{c:stochastic} and~\ref{c:gen} imply that $\{P(t)\}_{t\ge0}$ is stochastic and its generator is $(A_{\Psi}+B,\mathcal{D}(A_\Psi))$.
\end{proof}

\subsection{Proofs of Theorems \ref{t:Kkernel}--\ref{c:invconv}}\label{s:inv}

\begin{proof}[Proof of Theorem~\ref{t:Kkernel}]
First, we look more closely at the defining formula of the operator $K $ in \eqref{eq:K} when the operators $B$ and $\Psi$ are as given in \eqref{d:part}.
Suppose that $(f,f_{\partial})\in L^1(E,m)\times L^1(\nbound ,m^{-})$ are nonnegative.
Using monotonicity of $\lambda\mapsto R(\lambda,A_0)f$  and $\lambda\mapsto \Psi(\lambda)f_{\partial}$ we infer that
the pointwise limits
\begin{equation*}
R(0)f=\lim_{\lambda \to 0^+}R(\lambda,A_0)f \quad \text{and}\quad \Psi(0)f_{\partial}=\lim_{\lambda \to 0^+}\Psi(\lambda)f_{\partial}
\end{equation*}
exist and that $R(0)f,\Psi(0)f_{\partial}$ are nonnegative, but need not be integrable.
Since $\|\vartheta R(\lambda,A_0)f \|\le \|f\|$ and $\|\vartheta\Psi(\lambda)f_{\partial}\|\le \|f_{\partial}\|$ for each $\lambda>0$, by Lemmas~\ref{r:psil} and \ref{r:gamA0},  we see that
\begin{equation*}
R(0)f(x)=\int_{0}^{\nlife(x)}e^{-\int_0^t \vartheta(\phi_{-r }(x))dr} f(\phi_{-t}(x))J_{-t}(x)dt, \quad x\in E,
\end{equation*}
and
\begin{equation*}
\Psi(0)f_{\partial}(x)=\mathbf{1}_{\{\nlife(x)<\infty\}}e^{-\int_0^{\nlife(x)} \vartheta(\phi_{-r }(x))dr} f_\partial(\phi_{-\nlife(x)}(x))J_{-\nlife(x)}(x),\quad x\in E,
\end{equation*}
together with $\vartheta R(0)f, \vartheta \Psi(0)f_{\partial}\in L^1(E,m)$.
Similarly, $\|\gamma^{+} R(\lambda,A_0)f \|\le \|f\|$ and   $\|\gamma^{+}\Psi(\lambda)f_{\partial}\|\le \|f_{\partial}\|$ for all $\lambda>0$,  and we have $\gamma^+R(0)f, \gamma^{+}\Psi(0)f_{\partial}\in L^1(\pbound ,m^{+})$, where
\begin{equation*}
\gamma^{+}R(0)f(z)=\int_{0}^{\nlife(z)}e^{-\int_0^{t}\vartheta (\phi_{-r}(z))dr}f(\phi_{-t}(z))J_{-t}(z)dt,\quad z\in \pbound ,
\end{equation*}
and
\begin{equation*}
\gamma^{+}\Psi(0)f_{\partial}(z)=\mathbf{1}_{\{\nlife(z)<\infty\}}e^{-\int_0^{\nlife(z)} \vartheta(\phi_{-r}(z))dr} f_\partial(\phi_{-\nlife(z)}(z))J_{-\nlife(z)}(z),\quad z\in \pbound .
\end{equation*}
Consequently, for $R_0$ as in \eqref{d:R0p} we obtain
\begin{equation*}
R_0(f,f_{\partial})(x)=R(0)f(x)+\Psi(0)f_{\partial}(x), \quad x\in E,
\end{equation*}
and
\begin{equation*}
R_0(f,f_{\partial})(z)=\gamma^{+}R_0(f,f_{\partial})(z)= \gamma^+R(0)f(z)+\gamma^{+}\Psi(0)f_{\partial}(z), \quad z\in \pbound.
\end{equation*}
Thus, the operator $K$ as defined in \eqref{eq:K} is given by \eqref{d:operatorK}. Note that if  condition~\eqref{d:inflife} holds for all $x$ with $\life(x)=+\infty$ then, by \eqref{eq:T0St} and the dominated convergence theorem, the semigroup $\{S(t)\}_{t\ge 0}$ satisfies
\[
\lim_{t\to\infty}\|S(t)f\|=\lim_{t\to\infty}\int_{E} e^{-\int_{0}^{t}\vartheta(\phi_r(x))dr}\mathbf{1}_{[0,\life(x))}(t) |f(x)|\,m(dx)=0 , \quad f\in L^1(E,m),
\]
and it is thus  strongly stable.
Now Theorem~\ref{t:opK}, Lemma~\ref{l:ltr} and the monotone convergence theorem imply the result.
\end{proof}

\begin{proof}[Proof of Theorem~\ref{c:invsem}]
Let $(f,f_{\partial})$ be an invariant density for the operator $K$.
For $\overline{f}$ as in \eqref{eq:esd} and $B\in \mathcal{B}(E)$ we have
\[
\int_B \overline{f} dm=\int_E \lim_{\lambda\to 0^{+}}(R(\lambda,A_0)f+\Psi(\lambda)f_{\partial}) \mathbf{1}_{B}dm =\lim_{\lambda\to 0}\langle R(\lambda,\mathcal{A})(f,f_{\partial}), \mathbf{1}_{B}\rangle
\]
by \eqref{eq:RA} and the monotone convergence theorem. It follows from \eqref{e:U0lambda} and \eqref{d:Ul0} that
\[
\int_B \overline{f} dm=\int_E \int_0^{\life(x)}\mathbf{1}_{B}(\phi_t(x))e^{-\int_{0}^t \vartheta(\phi_r(x))dr}dt f(x)m(dx)+ \int_{\nbound}\int_0^{\life(x)}\mathbf{1}_{B}(\phi_t(x))e^{-\int_{0}^t \vartheta(\phi_r(x))dr}dtf_{\partial}(x)m^{-}(dx).
\]
Since  $\overline{f}=R_0(f,f_{\partial})$, we see that assumption  \eqref{a:intinv} gives $\overline{f}\in L^1(E,m)$. Consequently, the result follows from Theorem~\ref{thm:exunifp}.  Observe that condition \eqref{d:norma} holds as well.
\end{proof}

\begin{proof}[Proof of Theorem \ref{c:invconv}] We show that Theorem~\ref{thm:exunifpi} applies.
Let $f_0=\lambda f_*$ and
\[
f_N=R(\lambda,A_0)Bf_{N-1}+\Psi(\lambda)\Psi f_{N-1}, \quad N\ge 1.
\]
Then $f_N\in \mathcal{D}$ for $N\ge 1$ and $f_N\uparrow f_*$. By Lemmas ~\ref{r:psil} and \ref{r:gamA0}, we see that
\[
\begin{split}
\int_{\pbound }\gamma^{+} f_N dm^{+}+\int_E (\lambda+\vartheta) f_Ndm&=\int_{\nbound }\Psi f_{N-1} dm^{-}+\int_{E} B f_{N-1}dm.
\end{split}
\]
This together with Assumption~\ref{a:jump} implies that
\[
\int_{\pbound }\gamma^{+} f_Ndm^{+}+\int_E \vartheta  f_Ndm\le \int_{\pbound }\gamma^{+}f_{N-1}dm^{+}+\int_{E}\vartheta f_{N-1}dm
\]
for all $N$. Since $f_1\in \mathcal{D}$, we see that $Bf_*\in L^1(E,m)$, $\Psi f_*\in L^1(\nbound ,m^{-})$ and
\[
\|(B f_*,\Psi f_*)\|=\int_{\pbound }\gamma^{+} f_*dm^{+}+\int_E \vartheta f_*dm\le \int_{\pbound }\gamma^{+} f_1dm^{+}+\int_E \vartheta  f_1dm<\infty.
\]
Theorem~\ref{thm:exunifpi}  completes the proof.
\end{proof}

\section{Examples}\label{s:examples}

\subsection{Several flows}\label{s:ss} In this section we look at  the general setting considered by Davis \cite{davis84,davis93}.
Let $\gstate_i\subset \mathbb{R}^{d_i}$, $i\in I$, be a collection of open sets, where $I$ is a finite or a countable set,  such that on each set $\gstate_i$ there is a flow  $\phi_t^i\colon \gstate_i\to \gstate_i$, $t\in \mathbb{R}$, $i\in I$,
 defined  by solutions of the differential equation
\begin{equation}\label{d:flowbi}
\frac{d}{dt}x(t)=b_i(x(t)), \quad x(0)=x^0,
\end{equation} where $b_i$ is locally Lipschitz continuous. For each $i$ let $E_i^0\in \mathcal{B}(\mathbb{R}^{d_i})$ be  such that its closure $\overline{E}_i^{0}$ is contained in  $\gstate_i$.
We define
two subsets of the boundary of the set $E_i^0$: the outgoing boundary
\[
\pbound_i=\{z^0\in \overline{E}_i^0\setminus E_i^0: z^0=\phi_t^i(x^0)\text{ for some } x^0\in E_i^0, t>0, \text{ and } \phi_s^i(x^0)\in E_i^0 \text{ for } s\in [0,t)\}
\]
which are points which can be reached by the flow $\phi_t^i$  from $E_i^0$ in a finite positive time and the incoming boundary
\[
\nbound_i=\{z^0\in \overline{E}_i^0\setminus E_i^0: z^0=\phi_{-t}^i(x^0)\text{ for some } x^0\in E_i^0, t>0, \text{ and } \phi_{-s}^i(x^0)\in E_i^0 \text{ for } s\in [0,t)\} .
\]
We define $E_i=E_i^0\cup \nbound_i\setminus (\nbound_i\cap \pbound_i)$, $i\in I$,  $E^0=\{(x^0,i):x^0\in E_{i}^0,i\in I\}$,
and
the state space of the process by
\[
E=\{(x^0,i):x^0\in E_i, i\in I\} .
\]
The points from the sets
\[
\bound =\bigcup_{i\in I} \bound_i\times\{i\}
\]
can be reached by the flow from $E^0$ in a finite positive/negative time. For each $i$ we also consider a Borel measurable nonnegative function $\vartheta_i\colon E_i\to [0,\infty)$.

Let $\gstate=\bigcup_{i\in I} \gstate_i\times \{i\}$ and let $\widetilde{\mathcal{E}}$ be the $\sigma$-algebra which is the union of Borel $\sigma$-algebras  of subsets of $\gstate_i$. The space $\gstate$ can be endowed with a metric  in such a way that $\gstate$ is a separable metric space. We define
$\phi_t \colon \gstate\to \gstate$ by
\[
\phi_t(x)=(\phi_t^i(x^0),i), \quad x=(x^0,i),  x^0\in \gstate_i,i\in I.
\]
The mapping
$\mathbb{R}\times \gstate \ni(t,x)\mapsto\phi_t(x)\in \gstate$ is continuous and \eqref{e:ds} holds.  Thus $\phi$ is a flow on $\gstate$.
We  consider the
$\sigma$-finite measure $m$ on $\gstate$ given by
\[
m(\F )=\sum_{i\in I} (m_i\times \delta_i)(\F ),\quad \F \in \widetilde{\mathcal{E}},
\]
where $m_i$ is the Lebesgue measure on $\mathbb{R}^{d_i}$, $i\in I$, and the jump rate function given by
$\vartheta(x)=\vartheta_i(x^0)$ for $x=(x^0,i)$ with  $ x^0\in E_i$, $i\in I$.
We assume that the interior  of each set $E_i^{0}$ is non-empty and that the boundary of the set $E_i^0$ is of Lebesgue measure $m_i$ zero.

\begin{corollary}
Suppose that for each $i\in I$ the vector field $b_i\colon \gstate_i\to \mathbb{R}^{d_i}$ in \eqref{d:flowbi} is continuously differentiable with bounded derivative and that $\vartheta_i$ is continuous. Then Assumptions \ref{a:nons}--\ref{a:dive} hold true. If, additionally, a jump distribution $\jump$ is such that Assumption \ref{a:jump} holds then the process with characteristics $(\phi,\vartheta,\jump)$ induces a substochastic semigroup on $L^1(E,m)$.
\end{corollary}
\begin{proof}
From the theory of differential equations it follows that for each $i$ there is a flow on the set $\gstate_i$  defined by solutions of the initial value problem \eqref{d:flowbi}. If $m_i$ is the Lebesgue measure on $\mathbb{R}^{d_i}$ then the Jacobian $J_t^i$ of the flow $\phi^i$ is given by
\[
J_t^i(x^0)=\exp\left\{\int_0^{t} \dive(b_i(\phi_{r}^i (x^0)))dr\right\},\quad x^0\in \gstate_i,
\]
where $\dive(b_i(x^0))$ is the divergence of the vector field $b_i$. We define $
J_t(x)=J_t^i(x^0)$ for $x=(x^0,i)$ with $x^0\in \gstate_i$, $i\in I$, and we note that Assumption \ref{a:nons} holds.
Given $i$  the function  $x^0\mapsto \dive(b_i(x^0))$ is bounded and there exist  unique Borel measures $m_i^{\pm}$ such that condition \eqref{e:mpm} holds for the flow $\phi^i$ on $E_i$ with the corresponding boundaries $\bound_i$, by \cite{arlottibanasiaklods07}.
Therefore, Assumption \ref{a:dive} is satisfied if we consider the measures $m^{\pm}=\sum_{i\in I} m_i^{\pm}\times \delta_i$. Since for each $i$ the function $\vartheta_i$ is continuous we see that Assumption \ref{a:varp} also holds.
\end{proof}

\subsection{Kinetic equations with conservative boundary conditions}\label{s:kinetic}

In this section we provide the link between PDMPs and transport equations with boundary conditions; for the general treatment of the latter see \cite{voigt80,beals87,greenberg87,cercignani88,villani02,lods05} and the references therein.
We consider here a general time dependent linear kinetic problem  for a density $u$ depending on time $t$,  position $x\in \Omega$ and velocity $v\in V$, where  $\Omega\times V\subseteq \mathbb{R}^{2d}$. The movement is defined by the flow
given by the differential equation
\begin{equation}\label{e:vjp}
x'(t)=v(t),\quad v'(t)=0.
\end{equation}
 The solution of \eqref{e:vjp} with initial condition $(x(0),v(0))=(x,v)$ is of the form
\[
\phi_t(x,v)=(x+t v,v),\quad x\in \mathbb{R}^d, v\in V, t\in \mathbb{R}.
\]
We take $\gstate=\mathbb{R}^{2d}$,  $E^0=\Omega\times V$, and $m=\Leb\times \nu$, where $\nu$ is a Radon measure on $\mathbb{R}^d$ with support $V$.
We have
\[
 \bound =\{(x,v)\in \partial \Omega\times V: \pm v\cdot n(x) >0\},\quad m^{\pm}(dx,dv)=\pm v\cdot n(x)\sigma (dx)\nu(dv) ,
\]
where $n(x)$ is the outward normal at $x\in \partial \Omega$, and
$\sigma$ is the surface Lebesgue measure on the boundary $\partial \Omega$.
Supplementary conditions must be specified on the boundary of the phase space. We assume that they are modeled by a positive boundary operator $H$ relating the incoming and outgoing boundary fluxes of particles.
There is also given a collision frequency $\vartheta(x,v)$ and a collision kernel $k(x,v,v')$, which are  nonnegative measurable functions such that
\[
\int_V k(x,v',v)\nu(dv')=\vartheta(x,v),\quad (x,v)\in \Omega\times V.
\]
 Thus the equation for $u$ is of the form
\begin{equation*}
\frac{\partial u(t,x,v)}{\partial t}+v\cdot \nabla_{x} u(t,x,v)=\int_V k(x,v,v')u(t,x,v')\nu(dv')-\vartheta(x,v)u(t,x,v)
\end{equation*}
with boundary and initial conditions
\[
\gamma^{-}u=H(\gamma^{+} u), \quad u(0,x,v)=f(x,v).
\]
The boundary operator $H$ is assumed to have norm equal to 1.
Let the jump distribution $\jump$ be such that
\[
\vartheta(x,v)\jump((x,v), \F )=\int_V \mathbf{1}_{\F}(x,v')k(x,v',v)\nu(dv'),\quad (x,v)\in E^0,\F \in \mathcal{B}(E),
\]
and
\[
\int_{\pbound }\jump(z,\F )f_{\partial^+}(z)m^{+}(dz)=\int_{\nbound \cap \F }H(f_{\partial^+})(z)m^{-}(dz),\quad f_{\partial^+}\in L^1(\pbound ,m^{+}), \F \in \mathcal{B}(\nbound ).
\]
Thus, we have $P_0(f,f_{\partial^+})=P_0(f)$ and $P_{\partial}(f,f_{\partial^+})=H(f_{\partial^+})$ for $f\in L^1(E,m)$, $f_{\partial^+}\in L^1(\pbound,m^{+})$, where
\[
P_0(\vartheta f)(x,v)=\int_{V}k(x,v,v')f(x,v')\nu(dv')
\]
for $f\in L^1(E,m)$ such that $\vartheta f\in L^1(E,m)$.
If for each $v\in V$ the function $x\mapsto \vartheta(x,v)$ is locally integrable on $\Omega$, then the process with characteristics $(\phi,\vartheta,\jump)$ induces a substochastic semigroup on $L^1(E,m)$. Moreover, if $\vartheta$ is bounded and $\inf\{\life(x,v):(x,v)\in \nbound \}>0$ then the semigroup corresponding to  $(\phi,\vartheta,\jump)$ is stochastic and the kinetic equation  is well posed on $L^1(E,m)$, by Corollary \ref{c:sto}.

A particular example is the collisionless transport equation, where $k\equiv 0$ or, equivalently, $q\equiv 0$,  see \cite{voigt80,arlottilods05,mustapha17,lods2018invariant} and the references therein. Consider now the operator $K$ as in \eqref{d:operatorK}. We have $K(f,f_{\partial})=(0, H(R_0(f,f_{\partial})))$  for $(f,f_{\partial})\in L^1(E,m)\times L^1(\nbound , m^{-})$. Observe  that the operator $K $ has an invariant density $(f,f_{\partial})$ if and only if $f=0$ and $f_{\partial}$ is the solution of $H(R_0(0,f_{\partial}))=f_{\partial}$; in that case, the induced substochastic semigroup has an invariant density if $R_0(0,f_{\partial})\in L^1(E,m)$, by Theorem~\ref{c:invsem}.

\subsection{Application to a two phase cell cycle model}\label{s:cell}

In this section we give an example of a PDMP where   the induced semigroup is stochastic as in Corollary~\ref{c:sto} and its generator is the operator $A_{\Psi}+B$ but the jump rate function $\vartheta$ need not be bounded and $\inf\{\life(z):z\in \nbound \}=0$. Consider a continuous time version of the two-phase cell cycle model from \cite{TH-cc,Tyrcha,al-mcm-jt-92} as presented in \cite{rudnickityran15}.
We assume that the cell cycle consists of two phases: $I$ and $II$. The phase $I$ begins at birth and lasts until a critical event occurs which is necessary for mitosis and cell division. Then the cell enters the phase $II$ which lasts for a finite time $T_{II}$. We assume that a cell of size $x>0$ grows with rate $g(x)$, it enters the phase $II$ with rate $\varphi(x)$,
and at the end of the phase $II$ it splits into two daughter cells  with sizes $x/2$.

The model can be described as a piecewise deterministic Markov process. We consider three variables $(x,y,i)$, where $x$ describes the cell size, $y$ describes the time which elapsed since the moment the cell entered the phase $II$,   $i=1$ if a cell is in the phase $I$,  and $i=2$ if it is in the phase $II$.
Between jump points  the coordinates of the process
$X(t)=(x(t),y(t),i(t))$ satisfy the  following system of ordinary differential equations
\begin{equation}\label{e:ccde}
x'(t)=g(x(t)),\quad y'(t)=i(t)-1,\quad i'(t)=0.
\end{equation}
The generation time of a cell, i.e. the time from birth to division, is equal to $T_{I}+T_{II}$, where $T_{I}$ is the random length of the phase $I$ with distribution
\begin{equation}
 \mathbb{P}(T_{I}>t|x(0)=x)=e^{-\int_0^t \varphi(x(r))dr},\quad t\ge 0.
\end{equation} Let $t_0=0$. If consecutive descendants of a given cell are observed and the $n$th generation time is denoted by $t_n$, then $t_{n+1}=s_{n}+T_{II}$ where   $s_n$ is the time when the cell from the $n$th generation  enters the phase $II$,  $n\ge 0$. A newborn cell at time $t_n$ has an initial size equal to $x(t_n^-)/2$, where $x(t_n^-)$ is the size of its mother cell.  Thus
\begin{equation*}
x(s_n)=x(s_n^-), \quad i(s_n)=2,
\end{equation*}
and the cell divides into two cells at the end of the phase $II$, so that we have
\begin{equation*}\label{d:jumptn}
x(t_{n+1})=\frac12 x(t_{n+1}^-), \quad i(t_{n+1})=1.
\end{equation*}

We assume that $g\colon (0,\infty)\to (0,\infty)$ is a  continuous
 function  such that $g(x)>0$ for $x>0$ and
\[
\mathfrak{G}(\infty)=|\mathfrak{G}(0)|=\infty, \quad \text{where}\quad \mathfrak{G}(x)=\int_{\bar{x}}^x \frac{1}{g(y)}dy
\]
with  $\bar{x}>0$.
Observe that $\phi_t^1(x_0)=\mathfrak{G}^{-1}(\mathfrak{G}(x_0)+t)$ is the solution of $x'(t)=g(x(t))$ with $x(0)=x_0$.
The solution of \eqref{e:ccde} with initial condition $(x,y,i)$
 is given by
\[
\phi_t(x,y,i)=(\phi_t^1(x),y+(i-1)t,i),\quad x\in (0,\infty), y,t\in \mathbb{R}, i\in \{1,2\}.
\]
We take $\gstate=(0,\infty)\times \{0\}\times\{1\}\cup (0,\infty)\times \mathbb{R}\times \{2\}$ and
$
E^0=(0,\infty)\times\{0\}\times \{1\}\cup (0,\infty)\times(0,T_{II})\times \{2\}.
$
We have
\[
\nbound =(0,\infty)\times \{0\}\times \{2\}\quad \text{and}\quad \pbound =(0,\infty)\times \{T_{II}\}\times\{2\}.
\]
We introduce the measure
\[
m(\F )=(\Leb\times\delta_{0}\times \delta_{1})(\F )+(\Leb\times \Leb\times \delta_2)(\F ), \quad \F \in \mathcal{B}(\gstate),
\]
where $\Leb$ is the one dimensional Lebesgue measure.
Observe that
\[
J_t(x,y,i)=\frac{g(\phi_t^1(x))}{g(x)},\quad (x,y,i)\in \gstate, t\in \mathbb{R}.
\]
The measures at boundaries are taken to be
\[
m^{-}=\Leb\times\delta_{0}\times \delta_{2}\quad \text{and}\quad m^{+}=\Leb\times\delta_{T_{II}}\times \delta_{2}.
\]
The jump rate function $q$ is given by $\vartheta(x,0,1)=\varphi(x)$ and $\vartheta(x,y,2)=0$, $(x,y,i)\in E$.
We assume that the function $\varphi\colon (0,\infty)\to[0,\infty)$ is locally integrable on $(0,\infty)$. Finally, two types of jumps are possible: if $i=1$ then there is a jump from $(x,0,1)$ to $(x,0,2)$ with rate $\varphi(x)$, while if $i=2$ then the boundary $\pbound $ is reached in a finite time
and there is a forced jump from the point $(x,T_{II},2)$ to the point $(\frac12x,0,1)$. Observe that we have
\[
P_{0}(f,f_{\partial^+})(x,0,1)=2f_{\partial^+}(2x,T_{II},2), \quad P_{0}(f,f_{\partial^+})(x,y,2)=0,\quad  x>0, y\in (0,T_{II}),
\]
and
\[
P_{\partial}(f,f_{\partial^+})(x,0,2)=f(x,0,1),\quad x>0, f\in L^1(E,m), f_{\partial^+}\in L^1(\pbound ,m^{+}).
\]

The  operator $A$ can be interpreted as
\[
Af(x,0,1)=-\frac{\partial }{\partial
x}(g(x)f(x,0,1))-\varphi(x)f(x,0,1),\quad
Af(x,y,2)=-\frac{\partial }{\partial x}(g(x)f(x,y,2))-\frac{\partial }{\partial
y}(f(x,y,2)),
\]
where the derivatives are understood in the sense of distributions.
The operator $B\colon \mathcal{D}\to L^1(E,m)$ and the boundary  operator $\Psi\colon \mathcal{D}\to L^1(\nbound ,m^{-})$ are given by
\begin{equation}\label{d:BPsi}
 Bf(x,0,1)=2\gamma^{+}f(2x,T_{II},2),\quad Bf(x,y,2)=0, \quad
\Psi f(x,0,2)=\varphi(x)f(x,0,1),\quad x>0.
\end{equation}

\begin{corollary}
The induced semigroup $\{P(t)\}_{t\ge 0}$ corresponding to $(\phi,\vartheta,\jump)$ is stochastic and its generator is the operator $(A_{\Psi}+B,\mathcal{D}(A_\Psi))$, where $\mathcal{D}(A_\Psi)=\{f\in \mathcal{D}: \gamma^{-}f(x,0,2)=\varphi(x) f(x,0,1), x>0\}.$  \end{corollary}
\begin{proof} First we make use of  Lemma~\ref{l:APsi} to show that $A_\Psi$ is  resolvent positive. Since $\nlife(x,0,1)=\infty$ for $x>0$, we have $\Psi(\lambda)f_{\partial}(x,0,1)=0$ by \eqref{e:psilam}. Hence,
$
\Psi\Psi(\lambda)f_{\partial}=0
$
and the operator $I_{\partial}-\Psi\Psi(\lambda)$ is the identity.  Lemma~\ref{l:APsi}
now implies that $A_\Psi$ is resolvent positive and that
\begin{equation}\label{ex:tp}
R(\lambda,A_\Psi)f=R(\lambda,A_0)f+\Psi(\lambda)\Psi R(\lambda,A_0)f, \quad f\in L^1(E,m).
\end{equation}
For any nonnegative $f\in L^1(E,m)$, we have
\[
\|BR(\lambda,A_{\Psi})Bf\|=\int_0^\infty \gamma^{+}R(\lambda,A_{\Psi})Bf(x,T_{II},2)dx.
\]
This together with \eqref{ex:tp}  gives
\[
\|BR(\lambda,A_{\Psi})Bf\|=\int_{\pbound } \gamma^{+}R(\lambda,A_{0})Bf dm^{+}+\int_{\pbound } \gamma^{+}\Psi(\lambda)\Psi R(\lambda,A_0)Bf dm^{+}.
\]
Since $\life(x,0,1)=+\infty$  and  $Bf(x,y,2)=0$ for $x>0$, $y\in (0,T_{II})$,  it follows from Lemma~\ref{r:gamA0} that the integral of $\gamma^{+}R(\lambda,A_{0})Bf$ is zero. Lemma~\ref{r:psil} now implies that
\begin{align*}
\|BR(\lambda,A_{\Psi})Bf\|=\int_{\nbound }e^{-\lambda T_{II}}\Psi R(\lambda,A_0)Bf dm^{-}=e^{-\lambda T_{II}}\int_{0}^\infty \varphi(x)R(\lambda,A_0)Bf(x,0,1)dx.
\end{align*}
Observe that the last integral is smaller than $\|Bf\|$, by Remark~\ref{r:gamA0}.
Consequently, we obtain
\[
\|(BR(\lambda,A_{\Psi}))^2f\|\le e^{-\lambda T_{II}}\|BR(\lambda,A_{\Psi})f\|\le  e^{-\lambda T_{II}}\|f\|.
\]
Corollary~\ref{c:both} implies that the induced semigroup is stochastic and that its generator is $(A_{\Psi}+B,\mathcal{D}(A_\Psi))$.
\end{proof}

We close this section by looking at invariant densities for the corresponding operator $K $ as in \eqref{d:operatorK} and for the induced semigroup.
We see that $(f,f_{\partial})$ is invariant for the operator $K $  if and only if
\[
\begin{split}
f(x,0,1)&=2R_0(f,f_{\partial})(2x,T_{II},2), \quad f(x,y,2)=0, \\
f_{\partial}(x,0,2)&=\varphi(x) R_0(f,f_{\partial})(x,0,1),\quad  x>0, y\in (0,T_{II}),
\end{split}
\]
where $R_0$ as defined in \eqref{d:R0p}  is given by
\[
R_0(f,f_{\partial})(x,0,1)=\frac{1}{g(x)}\int_{0}^x e^{\sQ (z)-\sQ (x)} f(z,0,1)dz \quad \text{with }\sQ(x)=\int_{\overline{x}}^x \frac{\varphi(y)}{g(y)}dy, \quad x>0,
\]
and
\[
R_0(f,f_{\partial})(x,y,2)=\int_{0}^{y} f(\phi_{-t}^1(x),y-t,2)\frac{g(\phi_{-t}^1(x))}{g(x)}dt + f_{\partial}(\phi_{-y}^1(x),0,2)\frac{g(\phi_{-y}^1(x))}{g(x)}, \quad x>0, y\in [0,T_{II}].
\]
Hence,
\begin{equation*}
f(x,0,1)=2f_{\partial}(\phi_{-T_{II}}^1(2x),0,2)\frac{g(\phi_{-T_{II}}^1(2x))}{g(2x)}\quad\text{and}\quad  f_{\partial}(x,0,2)=\frac{\varphi(x)}{g(x)}\int_{0}^x e^{\sQ (z)-\sQ (x)} f(z,0,1)dz.
\end{equation*}
Observe that
\[
\int_0^\infty  f_{\partial}(x,0,2)dx
=\int_0^\infty f(z,0,1)dz.
\]
Thus, $(f,f_{\partial})$ is an invariant density for the operator $K $ if and only if $f(x,y,2)=0$ and $f_{1}(x)=f(x,0,1)$ is an invariant density for the operator $P_1$ on $L^1(0,\infty)$ given by
\begin{equation}\label{op:cc1} P_1 f_1(x)=-\int _{0}^{\lambda (x)} \dfrac
{\partial}{\partial x}\left (
    e^{\sQ (z)-\sQ (\lambda(x))} \right )
    f_1(z)\, dz, \quad f_1\in L^1(0,\infty), \quad \text{where }\lambda(x)=\phi_{-T_{II}}^1(2x).
\end{equation}
Consequently, for $\overline{f}=R_0(f,f_{\partial})$ as in \eqref{d:invsemi},  we obtain
\[
\overline{f}(x,0,1)
=\frac{1}{g(x)}\int_{0}^x e^{\sQ (z)-\sQ (x)} f_{1}(z)dz\quad \text{and}\quad \overline{f}(x,y,2)=\frac{\varphi(\phi_{-y}^1(x))}{g(\phi_{-y}^1(x))}\int_{0}^{\phi_{-y}^1(x)} e^{\sQ (z)-\sQ (\phi_{-y}^1(x))} f_{1}(z)dz,
\]
and if $\overline{f}$  is integrable, then the semigroup $\{P(t)\}_{t\ge 0}$ has an invariant density, by Theorem \ref{c:invsem}.

It follows from \cite{gackilasota90} that if
\begin{equation*}\label{eq:blas}
\liminf_{x\to \infty}\bigl( \sQ (\lambda(x))-\sQ (x)\bigr)>1
\end{equation*}
then $P_1$ as defined in \eqref{op:cc1} has a unique  invariant density and we denote it by $f_{1}$.
We have
\[
\|\overline{f}\|=\int_0^{\infty} \left(\int_{z}^{\infty} \frac{1}{g(x)}e^{\sQ (z)-\sQ (x)}dx+T_{II}\right)f_1(z)dz
\]
and
\[
\int_{z}^{\infty} \frac{1}{g(x)}e^{\sQ (z)-\sQ (x)}dx=\int_0^\infty \mathbb{P}(T_{I}>t|x(0)=z)dt= \mathbb{E}_z(T_{I}).
\]
Hence, $\overline{f}$ is integrable if and only if
\[
\int_{0}^{\infty}\mathbb{E}_z(T_{I})f_1(z)dz<\infty. 
\]

\appendix
\section{Substochastic semigroups for flows and the transport operator}\label{s:flow}
\label{s:appen}

\renewcommand{\appendixname}{}%

In this appendix we prove Theorems~\ref{l:Green0} and \ref{t:gs0t}.  We need some auxiliary results concerning the set $\mathcal{D}_{\max}$ and the transport operator $\Tm$ defined  in \eqref{d:Tm}. We consider a flow $\{\phi_t\}_{t\in \mathbb{R}}$ on $\gstate$   satisfying Assumption~\ref{a:nons} and a set $E=E^0\cup \nbound\setminus \nbound\cap\pbound$ with $E^0\in \mathcal{B}(\gstate)$ and $\bound$ defined as in \eqref{d:activep}, \eqref{d:activen}.
Since the cocycle $\{J_t\}_{t\in \mathbb{R}}$ satisfies \eqref{d:RND},  the change of
variables leads to
\begin{equation}\label{e:pitgen}
\int_{\gstate}\mathbf{1}_{\F}(\phi_t(x))f(x)m(dx)=\int_{\gstate}\mathbf{1}_{\F}(x)f(\phi_{-t}(x))J_{-t}(x)m(dx),\quad \F \in \mathcal{B}(\gstate),f\in L^1(\gstate,m).
\end{equation}
We note that if we define
\begin{equation*}
\widehat{\phi}(t)f(x)=f(\phi_{-t}(x))J_{-t}(x), \quad x\in \gstate, t\in \mathbb{R},
\end{equation*}
for any  Borel measurable function $f\colon \gstate\to \mathbb{R}$, then
$\{\widehat{\phi}(t)\}_{t\ge 0}$ is a stochastic semigroup on $L^1(\gstate,m)$, by \cite[Theorem 4.12]{rudnickityran17}, and we obtain the following result.

\begin{theorem}\label{t:semP0}
Let
\begin{equation}\label{d:semP0}
S_0(t)f(x)=
\mathbf{1}_{E}(\phi_{-t}(x))f(\phi_{-t}(x))J_{-t}(x),\quad x\in E, t>0, f\in L^1(E,m).
\end{equation}
Then $\{S_0(t)\}_{t\ge 0}$ is a substochastic semigroup on $L^1(E,m)$ and
\begin{equation}\label{eq:P0St}
\int_E \mathbf{1}_{[0,\life(x))}(t)\psi (\phi_t(x))f(x)\,m(dx)=\int_E \psi (x) S_0(t)f(x)\,m(dx)
\end{equation}%
for all nonnegative Borel measurable $\psi \colon E\to \mathbb{R}$ and nonnegative $f\in L^1(E,m)$.
\end{theorem}

Given  $\varrho\in L^1(0,\infty)$ we define
\begin{equation}\label{d:diamond}
[\varrho\diamond f](x)=\int_0^\infty \varrho(s)S_0(s)f(x)\,ds,\quad x\in E, f\in L^1(E,m),
\end{equation}
where $\{S_0(t)\}_{t\ge 0}$ is as in \eqref{d:semP0}. Since $\{S_0(t)\}_{t\ge 0}$  is a substochastic semigroup, we see that for any
$f\in L^1(E,m)$ we have
\[
\|\varrho\diamond f\|\le \int_{0}^{\infty} |\varrho(s)|\|S_0(s)f\|\,ds\le \|f\| \int_{0}^{\infty} |\varrho(s)|\,ds,
\]
showing that $\varrho\diamond f\in L^1(E,m)$.

\begin{lemma}\label{l:rhof} Suppose that $\varrho$ is continuously differentiable with $\varrho,\varrho'\in L^1(0,\infty)$ and  that $f\in L^1(E,m)$. Then $\varrho\diamond f\in \mathcal{D}_{\max}$,
\begin{equation}\label{e:chec}
\int_E [\varrho\diamond f](x)\frac{d(\psi  \circ \phi_t)}{dt}\Big|_{t=0}(x)m(dx)=-\int_E \left([\varrho'\diamond f](x)+\varrho(0)f(x)\right)\psi (x)m(dx)
\end{equation}
for all $\psi \in \mathfrak{N}$ and
\begin{equation}\label{e:Tmrhof}
\Tm[\varrho\diamond f]=-\varrho'\diamond f-\varrho(0)f.
\end{equation}
Moreover, if $f\in \mathcal{D}_{\max}$ then
\begin{equation}\label{e:rhoTmf}
\varrho\diamond \Tm f=\Tm[\varrho\diamond f].
\end{equation}
\end{lemma}
\begin{proof}
First observe that if $\beta$ is a bounded measurable function and $f\in L^1(E,m)$, then
\begin{equation}\label{e:rteta}
\int_{E}[\varrho\diamond f](x)\beta(x)m(dx)=\int_E f(x)\int_0^{\life(x)} \varrho(s)\beta(\phi_s(x))\,ds\, m(dx),
\end{equation}
since
\[
\begin{split}
\int_{E}\int_0^\infty \varrho(s)S_0(s)f(x)\,ds\,\beta(x)m(dx)
&=\int_0^\infty \varrho(s)\int_E f(x)1_{[0,\life(x))}(s)\beta(\phi_s(x))m(dx)\,ds\\
&=\int_E f(x)\int_0^{\life(x)} \varrho(s)\beta(\phi_s(x))\,ds\, m(dx),
\end{split}
\]
where we used \eqref{d:diamond}, \eqref{eq:P0St} and Fubini's theorem.

Now fix  $\psi \in \mathfrak{N}$ and take
\[
\beta(x)=\frac{d(\psi  \circ \phi_t)}{dt}\Big|_{t=0}(x), \quad x\in E.
\]
We have $\beta(\phi_s(x))=\frac{d}{ds}(\psi  (\phi_s(x)))$ for $0<s<\life(x)$. Hence, integration by parts leads to
\[
\int_0^{\life(x)} \varrho(s)\frac{d}{ds}(\psi  (\phi_s(x)))\,ds=\lim_{s\to \life(x)}\varrho(s)\psi (\phi_{s}(x))-\varrho(0)\psi (x)-\int_0^{\life(x)}\varrho'(s)\psi  (\phi_s(x))\,ds.
\]
If $\life(x)<\infty$ then the limit in the last equation is equal to zero,  since $\psi  $ has a compact support in  $E^0$. If $\life(x)=\infty$, then the limit is also zero, since $\psi $ is bounded and $\varrho(s)\to 0$ as $s\to \infty$. This together with \eqref{e:rteta} gives
\[
\begin{split}
\int_E [\varrho\diamond f](x)\beta(x)m(dx)
&=\int_E f(x)\left(- \int_0^{\life(x)}\varrho'(s) \psi  (\phi_s(x))\,ds-\varrho(0)\psi (x)\right) m(dx).
\end{split}
\]
Using again Fubini's theorem and condition \eqref{eq:P0St}, we see that
\[
\begin{split}
\int_E f(x) \int_0^{\life(x)}\varrho'(s)\psi (\phi_s(x))\,ds\,m(dx)
&=\int_0^\infty \varrho'(s)\int_E f(x)1_{[0,\life(x))}(s)\psi (\phi_s(x))\,m(dx)\,ds\\
&=\int_0^\infty \varrho'(s)\int_E S_0(s)f(x)\psi (x)\,m(dx)\,ds \\
&=\int_E \int_0^\infty \varrho'(s)S_0(s)f(x)\,ds \,\psi (x) \,m(dx) .
\end{split}
\]
Therefore \eqref{d:Tm} holds true, implying  \eqref{e:Tmrhof}.

Since $\Tm f\in L^1(E,m)$ for $f\in \mathcal{D}_{\max}$, it follows from \eqref{e:rteta} that
\[
\int_{E}[\varrho\diamond \Tm f](x)\psi (x)m(dx)=\int_E \Tm f(x)\int_0^{\life(x)} \varrho(s)\psi (\phi_s(x))\,ds\, m(dx)
\]
for all $\psi \in \mathfrak{N}$. Observe that the function
\[
\psi _1(x)=\int_0^{\life(x)} \varrho(s)\psi (\phi_s(x))\,ds,\quad x\in E,
\]
belongs to $\mathfrak{N}$ and that
\[
\frac{d}{dt}(\psi _1(\phi_t(x)))\Big|_{t=0}=-\int_0^{\life(x)} \varrho'(s)\psi (\phi_s(x))\,ds-\varrho(0)\psi (x),\quad x\in E.
\]
Making use of \eqref{e:Tmrhof} for $\psi _1$, we see that
\[
\int_{E}[\varrho\diamond \Tm f](x)\psi (x)m(dx)=\int_E  f(x)\left(-\int_0^{\life(x)} \varrho'(s)\psi (\phi_s(x))\,ds-\varrho(0)\psi (x)\right)\,ds\, m(dx),
\]
which completes the proof.
\end{proof}

We use the approach of \cite{arlottibanasiaklods09} to get the characterization of elements from $\mathcal{D}_{\max}$.
We recall that two elements $f_1,f_2$ of the space $L^1(E,m)$ are equal if they are equal $m$-almost everywhere, i.e. $m\{x\in E:f_1(x)\neq f_2(x)\}=0$, and we say that $f_2$ is a representative of $f_1$. The following extends the divergence-free case \cite[Theorem 3.6]{arlottibanasiaklods09}.
\begin{theorem}\label{thm:abl}  Suppose that Assumptions~\ref{a:nons} and \ref{a:dive} hold.
If $f\in \mathcal{D}_{\max}$ then there exists  a representative $f^\sharp$ of $f$ such that for $m$-a.e. $x\in E$ and any $-\nlife(x)<t_1\le t_2<\life(x)$ we have
\begin{equation}
f^{\sharp}(\phi_{t_1}(x))J_{t_1}(x)-f^{\sharp}(\phi_{t_2}(x))J_{t_2}(x)=\int_{t_1}^{t_2}\Tm f(\phi_s(x))J_s(x)\,ds.
\end{equation}
\end{theorem}

\begin{proof} We use  a similar type of argument to the one  in the proof of \cite[Theorem 3.6]{arlottibanasiaklods09}.  Consider, as in \cite{arlottibanasiaklods09}, a sequence $(\varrho_n)_{n\ge 1}$ of one dimensional mollifiers supported on $[0,1]$: for each $n$ the function $\varrho_n\colon \mathbb{R}\to[0,\infty)$ is of class $C^{\infty}$, $\varrho_n(s)=0$ if $s\not \in [0,1/n]$, and $\int_{0}^{1/n}\varrho_n(s)ds=1$.
Continuity of the function $s\mapsto S_0(s)f$ implies that for each $\varepsilon>0$ we can find an $s_0>0$ such that
\[
\|S_0(s)f-f\|\le \varepsilon
\]
for all $s\le s_0$, hence that
\[
\|\varrho_n\diamond f-f\|\le \int_0^{1/n}\varrho_n(s)\|S_0(s)f-f\|\,ds\le \varepsilon
\]
for all $n\ge 1/s_0$. This shows that
\begin{equation}
\lim_{n\to \infty}\|\varrho_n\diamond f-f\|=0,\quad f\in L^1(E,m).
\end{equation}
Lemma~\ref{l:rhof} with $\varrho=\varrho_n$ now gives
\begin{equation}
\Tm (\varrho_n\diamond f)=\varrho_n\diamond \Tm f,\quad  f\in \mathcal{D}_{\max},n\ge 1.
\end{equation}
The rest of the argument is similar to \cite{arlottibanasiaklods09}.
\end{proof}

Next, we can identify the generator of the semigroup $\{S_0(t)\}_{t\ge 0}$.

\begin{theorem}\label{t:pt0g} Suppose that Assumptions~\ref{a:nons} and \ref{a:dive} hold. Let $\{S_0(t)\}_{t\ge 0}$  be the substochastic semigroup from Theorem~\ref{t:semP0}. Then its generator $(\mathrm{T}_0 ,\mathcal{D}(\mathrm{T}_0))$ is given by
\[
\mathrm{T}_0f=\Tm f,\quad f\in \mathcal{D}(\mathrm{T}_0)=\{f\in \mathcal{D}_{\max}: \gamma^{-}f=0\}.
\]
\end{theorem}
\begin{proof}
First, we show that the operator $(\Tm,\mathcal{D}_{\max})$ is an extension of the generator $(\mathrm{T}_0,\mathcal{D}(\mathrm{T}_0))$ of the semigroup $\{S_0(t)\}_{t\ge 0}$. To this end let $f\in \mathcal{D}(\mathrm{T}_0)$, $\lambda>0$,  and $g=\lambda f-\mathrm{T}_0f$.
Since
\[
f(x)=\int_0^\infty e^{-\lambda t}S_0(t)g(x)dt
\]
for $m$-a.e. $x\in E$, we have $f=\varrho\diamond g $ with $\varrho(s)=e^{-\lambda s}$, $s\in \mathbb{R}_+$. Lemma~\ref{l:rhof} now implies that $f\in \mathcal{D}_{\max}$ and $\Tm f=-\varrho'\diamond g-\varrho(0)g=\lambda f-g=\mathrm{T}_0f$.

To show that $\gamma^{-}f=0$ observe that for $z\in \nbound $ and  $0<s<\life(z)$ we have
\[
f(\phi_{s}(z))J_s(z)=\int_0^{s}e^{-\int_{\tau}^{s}\lambda dr}g(\phi_{\tau}(z))J_{\tau}(z)d\tau \]
and the limit of the right hand side is zero as $s\to 0$.
To prove that $\mathcal{D}_{\max}\cap\Ker(\gamma^{-})\subset \mathcal{D}(\mathrm{T}_0)$ we apply Theorem~\ref{thm:abl} and we argue as in Step 3 of the proof of \cite[Theorem 4.1]{arlottibanasiaklods09}.
\end{proof}

\begin{proof}[Proof of Theorem~\ref{l:Green0}]
Let $f\in \mathcal{D}$  and  $g=(\lambda-\Tm)f$. We define $f_1=R(\lambda,\mathrm{T}_0)g$ and $f_2=f-f_1$.
We see that $\Tm f_2=\lambda f_2$ and $\gamma^{-}f=\gamma^{-}f_2$.
It follows from Lemma~\ref{r:psil} with $\vartheta\equiv 0$ and equation \eqref{e:Grb} that $f_2=\Psi(\lambda)\gamma^{-}f\in \mathcal{D}(\gamma^{\pm})$  and
\[
\int_{E}\lambda f_2(x)m(dx)=\int_{\nbound }\gamma^{-}f_2(z)m^{-}(dz)-\int_{\pbound }\gamma^{+}f_2(z)m^{+}(dz)
\]
We have $\gamma^{-}f_1=0$, $g=(\lambda-\mathrm{T}_0)f_1$, and Lemma~\ref{r:gamA0} with $\vartheta\equiv 0$ implies
\[
\int_{E}\lambda f_1(x)m(dx)=\int_{E}g(x)m(dx)-\int_{\pbound }\gamma^{+}f_1(z)m^{+}(dz).
\]
Thus $\gamma^{+}f_1\in L^1(\pbound ,m^{+})$. Since $g=\lambda f -\Tm f$,  equality \eqref{e:Afin0} holds, by linearity.
\end{proof}

\begin{proof}[Proof of Theorem~\ref{t:gs0t}]
Theorem~\ref{t:semP0} and Assumption~\ref{a:varp} imply that $\{S(t)\}_{t\ge 0}$ is a substochastic semigroup on $L^1(E,m)$ and
that \eqref{eq:T0St} holds.
First we  show that the operator $(A,\mathcal{D})$ is an extension of the generator $(A_0,\mathcal{D}(A_0))$ of the semigroup $\{S(t)\}_{t\ge 0}$.
Let $f\in \mathcal{D}(A_0)$, $\lambda>0$ and   $g=\lambda f-A_0f$. We have $f=R(\lambda,A_0)g$ and  $\vartheta f\in L^1(E,m)$, by Lemma~\ref{r:gamA0}.
Arguing as  in the proof of Lemmas~\ref{r:gamA0}  and~\ref{l:rhof} it is easily seen that for each $\psi \in \mathfrak{N}$ we have
\[
\int_E (\lambda f(x)+\vartheta(x) f(x)-g(x))\psi (x)m(dx)=\int_E f(x)\frac{d(\psi  \circ \phi_t)}{dt}\Big|_{t=0}(x)m(dx).
\]
Thus we get $f\in \mathcal{D}_{\max}$ and $\Tm f=\lambda f +\vartheta  f -g$ showing that $A_0f=\Tm f-\vartheta f$.
Finally note that for $z\in \nbound $ and $s<\life(z)$ we have
\[
f(\phi_{s}(z))J_s(z)=\int_0^{s}e^{-\int_{\tau}^{s}(\lambda+\vartheta(\phi_{r}(z)))dr}g(\phi_{\tau}(z))J_{\tau}(z)d\tau, \]
implying that
$\gamma^{-}f=0$.
Consequently, we obtain
\[
A_0f=Af=\mathrm{T}_0f-\vartheta f,\quad f\in \mathcal{D}(\mathrm{T}_0)\cap L^1_{\vartheta},
\]
where $ L^1_{\vartheta}=\{f\in L^1(E,m): \vartheta f\in L^1(E,m)\}.$
The operator $(A,\mathcal{D}(\mathrm{T}_0)\cap L^1_{\vartheta})$ is dissipative as a sum of two dissipative operators.
Since $(A_0,\mathcal{D}(A_0))$ is the generator of a substochastic semigroup, we conclude that $(A_0,\mathcal{D}(A_0))=(A,\mathcal{D}(\mathrm{T}_0)\cap L^1_{\vartheta})$.
\end{proof}

%
%
\end{document}